\documentclass[12pt, a4paper]{amsart}
\usepackage{amscd,amsmath,amssymb,amsthm,amsfonts}
\usepackage[alphabetic]{amsrefs}
\usepackage{mathtools,mathrsfs,yfonts}
\usepackage[shortlabels]{enumitem}
\usepackage[all]{xy}
\usepackage{xcolor}
\usepackage[hypertexnames=true, citecolor = green, colorlinks = false, linkcolor = red, linktoc = none, pdffitwindow = false, urlbordercolor = white]{hyperref}%

\def\ker{\operatorname{ker}}

\def\id{\operatorname{id}}

\def\min{\operatorname{min}}

\def\id{\operatorname{id}}

\def\kms{\operatorname{KMS}}
\def\Irr{\operatorname{Irr}}

\def\C{\mathbb{C}}

\def\R{\mathbb{R}}
\def\N{\mathbb{N}}
\def\Z{\mathbb{Z}}

 \newcommand{\IZ}[0]{\mathbb{Z}}

\newcommand{\ia}[0]{\mathfrak{a}} \newcommand{\ib}[0]{\mathfrak{b}}
\newcommand{\ic}[0]{\mathfrak{c}}

 \renewcommand{\CD}[0]{\mathcal{D}}
 
 \newcommand{\CH}[0]{\mathcal{H}}
 
 \newcommand{\CL}[0]{\mathcal{L}}
 
\newcommand{\CO}[0]{\mathcal{O}} 
\newcommand{\CQ}[0]{\mathcal{Q}} 
 \newcommand{\CT}[0]{\mathcal{T}}

\newcommand{\FM}[0]{\mathfrak{M}} \newcommand{\FN}[0]{\mathfrak{N}}
 \newcommand{\FP}[0]{\mathfrak{P}}

\def\ssubset{\subset\hspace*{-1.7mm}\subset}

\newtheorem{thm}{Theorem}[section]
\newtheorem{corollary}[thm]{Corollary}
\newtheorem{lemma}[thm]{Lemma}

\newtheorem{proposition}[thm]{Proposition}

\theoremstyle{definition}
\newtheorem{definition}[thm]{Definition}
\newtheorem{notation}[thm]{Notation}

\theoremstyle{remark}
\newtheorem{remark}[thm]{Remark}

\newtheorem{example}[thm]{Example}

\newtheorem{question}[thm]{Question}

\numberwithin{equation}{section}

\begin{document}

\title[$C^*$-algebras of right LCM monoids and their equilibrium states]{$C^*$-algebras of right LCM monoids and \\ their equilibrium states}

\author[N.~Brownlowe]{Nathan Brownlowe}
\address{School of Mathematics and Statistics \\ University of Sydney\\ Australia}
\email{Nathan.Brownlowe@sydney.edu.au, Jacqui.Ramagge@sydney.edu.au}

\author[N.S.~Larsen]{Nadia S.~Larsen}
\address{Department of Mathematics \\ University of Oslo \\ P.O. Box 1053 \\ Blindern \\ NO-0316 Oslo \\ Norway}
\email{nadiasl@math.uio.no, nicolsta@math.uio.no}

\author[J.~Ramagge]{Jacqui Ramagge}
%\address{Department of Mathematics and Statistics \\ University of Sydney \\ Australia}
%\email{ramagge@sydney.edu.au}

\author[N.~Stammeier]{Nicolai Stammeier}
%\address{Department of Mathematics \\ University of Oslo \\ Norway}
%\email{nicolsta@math.uio.no}

\thanks{This research was supported by the Research Council of Norway (RCN FRIPRO 240362), the Australian Research Council (ARC DP170101821) and the Bergen Research Foundation through the project ``Pure Mathematics in Norway''.}

\begin{abstract}
We study the internal structure of $C^*$-algebras of right LCM monoids by means of isolating the core semigroup $C^*$-algebra as the coefficient algebra of a Fock-type module on which the full semigroup $C^*$-algebra admits a left action. If the semigroup has a generalised scale, we classify the KMS-states for the associated time evolution on the semigroup $C^*$-algebra, and provide sufficient conditions for uniqueness of the KMS$_\beta$-state at inverse temperature $\beta$ in a critical interval.
\end{abstract}

\date{\today}%{6 February 2019}
\maketitle

\section{Introduction}
The role of the $C^*$-algebra $C^*(G)$ associated to a discrete group $G$ is ubiquitous and fundamental in operator algebras.  Likewise, there are indispensable examples of $C^*$-algebras naturally associated to semigroups through representations by isometries, such as the Toeplitz algebra $\CT$ generated by a single isometry. While these prominent examples are well documented, a systematic theory built around $C^*$-algebras $C^*(S)$ associated to broad and abstract classes of discrete monoids $S$ has only recently emerged, largely as a consequence of an accelerating supply of new classes of monoids with good properties, viewed from a $C^*$-algebraic perspective. The example-driven insights have often brought about understanding of general properties, and vice versa. The present paper aims to strengthen our understanding of the interplay between right LCM monoids and their $C^*$-algebras by drawing on a varied array of examples, including some that are unexplored in the literature.

Before we explain our aim in more detail, we review recent developments around semigroup $C^*$-algebras. We refer to \cite{CELY} for an excellent introduction, and we restrict ourselves to highlighting a non-exhaustive list of recent work in this area. As a word of warning on terminology,  operator algebraists tend to not make a distinction between semigroups and monoids and casually refer to semigroup $C^*$-algebra in all situations. We shall adopt this habit, but since we are aware of big structural differences between the unital and the non-unital case in semigroup theory, we will specify monoid in all cases where the semigroup has an identity.

It was Nica's work in \cite{Nic} that decisively put forward $C^*$-algebras associated to nonabelian semigroups as a useful and versatile construction in operator algebras. Later work by Laca--Raeburn \cites{LR1,LR2} and Cuntz \cite{CuntzQ} provided new thrust in the form of new classes of $C^*$-algebras with interesting generating families of isometries and tractable properties. A general construction of full and reduced $C^*$-algebras for left cancellative monoids was introduced by Li, \cite{Li1}, and soon fundamental questions about  the interplay between nuclearity of the $C^*$-algebra and amenability of the monoid or its left inverse hull were raised, see \cites{Li2, Nor1}.  Subsequent work on semigroup $C^*$-algebras that centered on computing K-theory revealed impressively rich connections to number theory, see \cites{Li3, KT2}, and geometric group theory, see \cite{ELR}.

While  monoids in so-called quasi-lattice ordered pairs introduced by Nica dominated much of the early endeavours, the larger class of right LCM (for Least Common Multiple) monoids has in the past years received growing attention: The perspective from full $C^*$-algebras of right LCM monoids in a Zappa-Sz\'{e}p products enabled the authors of \cite{BRRW} to unify different constructions of $C^*$-algebras from \cites{LR2,Nek2,LRRW}. The structure of full and reduced semigroup $C^*$-algebras and natural quotients for right LCM monoids were analysed further in \cites{BLS1,Star,BS1,BLS2,BOS1}. More recent work indicates that there are more directions to be explored in this context, see for example \cites{B-Li,LOS,ACRS}.

Supplementing the study of semigroup $C^*$-algebras from the point of view of amenability, nuclearity, and K-theory, there has been intense activity on describing KMS-states (in honour of Kubo and Martin--Schwinger) for certain natural flows on the $C^*$-algebras, see \cites{LR2,BaHLR, CaHR} as well as the very recent \cites{ABLS,BLRS,ALN}.

The philosophy of unifying different case studies of $C^*$-algebras in a broad framework is the leitmotiv in \cite{ABLS}, where a first step was taken towards uncovering  internal structure of semigroup $C^*$-algebras via classification of KMS-states for what could be deemed as intrinsic one-parameter groups on $C^*(S)$.  The main result of \cite{ABLS} provides general methods for analysing KMS-states for $C^*$-algebras associated to \emph{admissible} right LCM monoids. This covers new cases such as algebraic dynamical systems \cite{BLS2}, and also offers a unified perspective onto this problem for $C^*$-algebras associated to the affine semigroup over the natural numbers \cite{LR2},  algebraic number fields with trivial class number \cite{CDL}, integer dilation matrices \cite{LRR}, self-similar group actions \cite{LRRW}, and quasi-lattice ordered Baumslag--Solitar monoids \cite{CaHR}.

The present work was triggered as we became aware of Spielberg's work \cite{Spi1} on Baumslag--Solitar monoids that are not quasi-lattice ordered, but right LCM. This class went unnoticed until after a discussion with Jack Spielberg at a conference in Newcastle, Australia, where the authors of the present paper realised that it fell outside the scope of \cite{ABLS}. The current paper grew out of our attempts to gain insight into what happens for this special class of Baumslag--Solitar monoids. For reasons that will become clearer later, we refer to these as Baumslag--Solitar monoids with \emph{positivity breaking}.

In this work we develop the full potential of classifying KMS-states for $C^*$-algebras of right LCM monoids as a means to uncover structural properties of  $C^*(S)$. We employ new techniques, which will allow us to distil the admissibility assumption of \cite{ABLS} to its indispensable part, namely the existence of a generalised scale which gives rise to the intrinsic one-parameter group on $C^*(S)$. Roughly speaking, a generalised scale is a special monoidal homomorphism from $S$ to the monoid $\N^\times$ of positive integers (with multiplication). The right LCM monoids covered in \cite{ABLS}*{Definition~3.1} are characterised by four admissibility conditions: the first two guarantee nice factorisation  properties for the right LCM monoid $S$ in question. The latter two conditions characterise the existence of a generalised scale. All the examples for right LCM monoids in \cites{LR2,LRR,BRRW,CaHR} suggested that these admissibility assumptions were quite modest. However, as discovered in Newcastle, the Baumslag--Solitar monoids with {positivity breaking} fail one of the two basic factorisation properties in an extreme way, see Section~\ref{subsec:BS3} for details.

The  factorisation properties of an admissible right LCM monoid $S$ are formulated with respect to its core submonoid $S_c := \{ a \in S \mid aS\cap sS \neq \emptyset \text{ for all } s \in S\}$, first considered for right LCM monoids in \cite{Star} with roots in \cite{CrispLaca}, and a counterpart of core irreducible elements $S_{ci}$, see \cite{ABLS}*{Subsection~2.1}. Faced with the absence of these factorisation properties for Baumslag--Solitar monoids with positivity breaking, we revert to techniques which stand in stark contrast to the ones in \cite{ABLS}. Our current analysis relies  on the study of the quotient space $S/_\sim$, where $s\sim t$ if $sa=tb$ for some $a,b \in S_c$, and the action $\alpha\colon S_c \curvearrowright S/_\sim$ induced by left multiplication. With these ingredients at hand, we develop a framework which will encompass admissible right LCM monoids as well as the Baumslag--Solitar monoids with {positivity breaking}, but also new examples, especially from self-similar group actions from virtual endomorphisms that may not exhibit the finite-state property.

Our approach to classifying KMS-states depends crucially on a special representation of $C^*(S)$, arising from the construction of a Fock-type right-Hilbert $C^*(S_c)$-module $\FM$ that we call the \emph{core Fock module}. We believe that this module is of independent interest as an abstract construction for arbitrary right LCM monoids, or equivalently, $0$-bisimple inverse monoids. %, see \cite{Law}.
Given the success of the induction process for classifying KMS-states on Pimsner algebras for ordinary Hilbert bimodules from \cite{LaNe1}, it is not surprising that a Fock module would be useful. The striking feature of the core Fock module is that it exists for every right LCM monoid. Its predecessor was constructed in \cite{ABLS}*{Theorem~8.4} using the factorisation properties encoded in the admissibility of $S$, which we now prove to be superfluous as we provide a  substantially refined and streamlined general construction.

To give a clue about the challenge for the new construction of $\FM$,  let us point out that admissibility of $S$ provides minimal elements for the equivalence classes of $S/_\sim$, similar to the natural numbers. Now, the Baumslag--Solitar monoids with positivity breaking are not admissible and so for these we must consider totally ordered equivalence classes that do not have minimal elements, similar to the integers.
%{In general, the situation may be a level worse, as there is no reason to infer that the equivalence classes would be directed with respect to the partial order we need to consider, that is, $s\leq_r t$ for $s\sim t$ if $s \in tS_c$.}
Thus, all the techniques in \cite{ABLS} that exploited the existence of minimal representatives must be substituted by solutions that would work with very little extra structure, namely, that equivalence classes are directed with respect to the partial order given by $s\geq_r t$ if $t \in sS_c$.

A further key asset of this work in comparison with \cite{ABLS} is the precision with which we identify when there is a unique KMS-state at inverse temperatures from the critical interval. Our main result in this respect, Theorem~\ref{thm:KMS results}, is a generalisation of \cite{ABLS}*{Theorem~4.3}. The theorem is an improvement in two ways: admissibility is replaced by the existence of a generalised scale, and the sufficiency criteria for uniqueness of the $\kms_\beta$-state for $\beta$ in the critical interval $[1,\beta_c]$ are weakened and unified at the same time. They appear to be very common and may in fact also be necessary conditions for uniqueness in many cases. %Uniqueness of KMS-states for inverse temperatures in the critical interval will be abstractly described upon measuring the absorbing elements in relation to the fixed points for $\alpha\colon S_c \curvearrowright S/_\sim$ with regards to the grading on $S/_\sim$ coming from the generalised scale.

The essential idea behind these uniqueness conditions is a detailed analysis of the internal structure of $C^*(S)$ by means of the core submonoid $S_c$. Roughly speaking, these conditions can be explained as follows: the generalised scale $N$ gives rise to a grading on $S/_\sim$ over the directed right LCM monoid $N(S)$. This provides a direction so that we can consider limits, and allows us to measure proportions of special equivalence classes among those that relate to a given $n\in N(S)$. For instance, we can look at fixed points for $a,b\in S_c$ under $\alpha\colon S_c \curvearrowright S/_\sim$ at level $n\in N(S)$, that is, $[s] \in S/_\sim$ with $s \in N^{-1}(n)$ and $asc= bsd$ for some $c,d \in S_c$. But it also makes sense to consider stronger fixed points, points that we call absorbing as they satisfy $asc=bsc$ for some $c \in S_c$. In other words, there is $s'\sim s$ with $as'=bs'$. The uniqueness criteria state, with varying particularities, that the proportion of non-absorbing fixed points at level $n\in N(S)$ tends to zero as $n\to\infty$.

In the case of faithful self-similar group actions, our condition can be given a measure-theoretic reformulation: The group acts on the space of infinite paths (starting from the root of the regular tree), which is equipped with the Borel probability measure determined by uniform distributions among the finite paths of equal length. Then the uniqueness condition for $\beta=1$ holds if and only if the boundary of the set of fixed points has measure zero, see Proposition~\ref{prop:almost surely regular}. The topological analogue, namely that the complement of this set is a dense $G_\delta$-set, is easily seen to be true for every faithful self-similar group action. But the measure theoretic question is much more subtle, and there may be room for rather special specimen, which would, if existent, surely be of interest on their own.

At $\beta=1$ we discover a  characterisation of the $\kms_1$-states  which could be summarised by saying that they correspond to the \emph{thermally inert traces}. Explicitly, we construct the $\kms_\beta$-states on $C^*(S)$ for $\beta>1$ via induction $\tau \mapsto \psi_{\beta,\tau}$ from normalised traces $\tau$ on $C^*(S_c)$ using the core Fock module. The $\kms_\beta$-state $\psi_{\beta,\tau}$ in turn restricts to a normalised trace on $C^*(S_c)$, once we precompose it with the $*$-homomorphism $\varphi\colon C^*(S_c)\to C^*(S)$ induced by the inclusion $S_c\subset S$. In this way, each inverse temperature $\beta>1$ comes with a self-map $\chi_\beta$ of the trace simplex on $C^*(S_c)$. The analogue of such a self-map in the setting of finite-state self-similar actions of groupoids on graphs was recently shown to exhibit intriguing dynamical features, see \cite{CSi}: the unique KMS-state is a universal attractor with regards to the dynamics defined by $\chi_\beta$. Here we take a different route, namely we show in Theorem~\ref{thm:restricting to tau} that the $\kms_1$-states correspond precisely to the common fixed points of $\{ \chi_\beta \mid \beta>1\}$. We further show that the simplex of common fixed points of $\{ \chi_\beta \mid \beta>\beta_0\}$ for every $\beta_0 \in [1,\infty)$ embeds into the simplex of $\kms_{\beta_0}$-states, which reveals a new continuity feature of the system $(C^*(S),\sigma)$.

We  begin this article with a small background section, followed by a presentation of the main results on classification of KMS-states: Theorem~\ref{thm:KMS results} and Theorem~\ref{thm:restricting to tau}. In Section~\ref{sec:ingredients} we start on the proof of the parametrisation part from Theorem~\ref{thm:KMS results}. In particular, we prove our promised existence of a core Fock module $\FM$ and an associated $*$-homomorphism from $C^*(S)$ into $\CL(\FM)$, see Theorem~\ref{thm:core Fock module}. Section \ref{sec:param of KMS-states} provides a parametrisation of KMS-states using $\FM$-induced representations on $C^*(S)$ from GNS representations of $C^*(S_c)$ coming from normalised traces. In this section we also prove Theorem~\ref{thm:restricting to tau}.
In Section~\ref{sec:uniqueness} we motivate our notions of core regular and summably core regular monoids, and prove that in different contexts they are sufficient conditions for achieving a unique $\kms_\beta$-state for $\beta$ in the critical interval.  Here we also complete the proof of Theorem~\ref{thm:KMS results}. Sections~\ref{subsec:BS3}--\ref{subsec:q-a but not almost free} are devoted to applications. In Section~\ref{subsec:BS3}, we illustrate the features around the phenomenon of positivity breaking in Baumslag--Solitar monoids $BS(c,d)^+$ with opposite signs for $c$ and $d$.

Conditions for  uniqueness of KMS$_\beta$-states in a critical interval bearing the flavour of ergodic theoretic considerations will be illustrated with examples arising from virtual endomorphism of the discrete Heisenberg group as in \cite{BK}, see Proposition~\ref{prop:almost surely regular} and Corollary~\ref{cor:applying BK}.

With the monoid of so-called shadowed natural numbers we provide the first examples of right LCM monoids whose $C^*$-algebras possess a unique $\kms_1$-state, despite the fact that the  action $\alpha:S_c\curvearrowright S/_\sim$ is not faithful, see Proposition~\ref{prop:unfaithful, but neg q-a}. In view of Proposition~\ref{prop:r canc: neg q-a forces alpha faithful}, uniqueness of the KMS$_1$-state is only possible due to failure of right cancellation. This degree of freedom was unavailable in other contexts, such as the finite, strongly-connected higher-rank graph case in \cite{aHLRS}, whose \emph{periodicity group} may be thought of as a measure for faithfulness of the action corresponding to $\alpha$.

Finally, we also provide a new context for characterising uniqueness of the KMS-states in the critical interval with two examples that fail almost freeness of the action $\alpha$, a property crucial in \cite{ABLS} as well as in several of the supporting case studies. In both examples the monoids are core regular and summably core regular, see Proposition~\ref{prop:ZxZtimes} and Proposition~\ref{prop:shift space with permutations}.

\vskip 0.2cm
\noindent{Acknowledgement:}
This research was initiated in the aftermath of the workshop ``Interactions between semigroups and operator algebras" July 24-27, 2017, at the University of Newcastle, Australia, supported by AMSI (Australian Mathematical Science Institute), and the last three authors thank the organisers, N. Brownlowe and D. Robertson, for the invitation. N.L. and N.S. also acknowledge the hospitality of the other two authors during a  visit to the University of Sydney. The authors are grateful to Volodia Nekrashevych for valuable discussions on the notion of $G$-regularity in the context of self-similar actions and virtual endomorphisms.

\section{Preliminaries}\label{sec:prelim}
\subsection{Right LCM monoids}

Let $S$ be a right LCM monoid, by which we mean a left cancellative semigroup $S$ with identity $1$ so that for every $s,t\in S$, we have $sS\cap tS=rS$ for some $r\in S$ whenever $sS\cap tS\neq\emptyset$. We shall adopt the notation $s\Cap t$ from \cite{Spi1} to indicate that $sS\cap tS\neq \emptyset$, and we write $s\perp t$ when $sS\cap tS=\emptyset$, where $s,t\in S$.

The \emph{core} of $S$ is the right LCM submonoid $S_c := \{a \in S\mid \forall s \in S: aS\Cap sS\}$, see \cite{Star}. Two elements $s,t \in S$ are said to be \emph{core equivalent}, denoted $s \sim t$, if there are $a,b \in S_c$ satisfying $sa=tb$. The class of $t\in S$  will be denoted $[t]$. We recall from \cite[Lemma 3.9]{ABLS} that left multiplication induces an action $\alpha$ of the semigroup $S_c$ by  bijections of $S/_\sim$, so $\alpha_a([t])=[at]$ for $a\in S_c$ and $t\in S$.

According to \cite{BRRW}, a finite subset $F$ of $S$, denoted $F\ssubset S$, is a foundation set if for every $s \in S$ there is $f \in F$ with $sS \Cap fS$. A foundation set $F$ is accurate, cf. \cite{BS1}, if $f\perp f'$ holds for all $f,f'\in F$ with $f\neq f'$.

We will be interested in right LCM monoids $S$ that admit a \emph{generalised scale}.

\begin{definition}\label{def:gen-scale}(\cite[Definition 3.1]{ABLS}) 
A generalised scale for a right LCM monoid $S$ is a nontrivial homomorphism of monoids $N\colon S \to \N^\times$ such that $\lvert N^{-1}(n)/_\sim\rvert = n$ for all $n \in N(S)$, and for each $n \in N(S)$, every transversal of $N^{-1}(n)/_\sim$ is an accurate foundation set for $S$.
  \end{definition}
  
  Although there is a wide range of examples of right LCM monoids which admit a generalised scale,  its existence  puts the focus on monoids with a particularly structured growth behaviour. This abstract assumption was given an explicit, combinatorial characterisation in \cite{Sta5}*{Theorem~3.11}.

In the sequel it will be useful to have a reference for the following simple observation.

\begin{lemma}\label{lem:consequence-ABLS-N-homomorphism}
Suppose that $S$ is a right LCM semigroup with a generalised scale $N$. If $s,t$ in $S$ satisfy $s\Cap t$ and $N_s\in N_tN(S)$, then $sS\cap tS=saS$ for some $a\in S_c$.
\end{lemma}
\begin{proof}
Assume $s,t\in S$ satisfy $sS\cap tS=ss'S$ with $ss'=tt'$ for some $s',t'\in S$. By \cite{ABLS}*{Proposition~3.6(iv)}, $N_{ss'}$ is the right LCM of $N_s$ and $N_t$ in $N(S)$. So if $N_s \in N_tN(S)$, then we must have $N_{ss'}=N_{s}$, that is, $N_{s'} =1$. Thus $s'\in S_c$ by \cite{ABLS}*{Proposition~3.6(i)}.
\end{proof}

Given $S$ and $N$ we define a partition  function to be $\zeta(\beta) := \sum_{[s] \in S/_\sim} N_s^{-\beta}$ with $\beta \in \R$. The \emph{critical inverse temperature} $\beta_c$ is the smallest $\beta_0 \in \R\cup\{\infty\}$ such that $\zeta(\beta)$ converges for all $\beta > \beta_0$; thus $\beta_c=\infty$ refers to divergence of $\zeta$ on the real line.

Recall from \cite{Sta5}*{Proposition~2.1} that every generalised scale $N$ satisfies condition (A4) from \cite{ABLS}*{Definition~3.1}, that is, $N(S)$ is freely generated by its irreducible elements. For $I \subset \text{Irr}(N(S))$, we let $S_I := N^{-1}(\langle I \rangle)\subset S$ denote the submonoid of $S$ corresponding to $I$, and remark that $S_I$ is also a right LCM semigroup due to \cite{ABLS}*{Proposition~3.6(iv)}.  The $I$-restricted $\zeta$-function as introduced in  \cite{ABLS}*{Definition~4.2} is given by
\begin{equation}\label{eq:restricted-zetaI}
\begin{array}{c}
\zeta_I(\beta) = \sum\limits_{[t] \in S_I/_\sim} N_t^{-\beta} = \sum\limits_{n \in \langle I \rangle}\sum\limits_{[t] :  N_t=n} N_t^{-\beta} = \sum\limits_{n \in \langle I \rangle} n^{1-\beta},
\end{array}
\end{equation}
and admits a product formula $\zeta_I(\beta) = \prod_{p \in I} (1-p^{1-\beta})^{-1}$ by
\cite{ABLS}*{Remark~7.4}. For a finite subset $F$ of $I$, denoted $F\ssubset I$, we set $m_F := \prod_{n \in F} n$ and remark that, due to (A4), this is the right LCM of $F$. The finite subsets of $I$ form a directed set with respect to inclusion, that is, $F\leq F' \Leftrightarrow F\subset F'$.

\subsection{\texorpdfstring{$C^*$-algebras of right LCM monoids}{C*-algebras of right LCM monoids}}

A general construction of \cite{Li1} associates to any left cancellative monoid $S$ a full semigroup $C^*$-algebra $C^*(S)$. In the present work we concentrate on right LCM monoids, in which case the presentation of $C^*(S)$ is somewhat simpler. The $C^*$-algebra $C^*(S)$ is the universal $C^*$-algebra generated by isometries $v_s,s\in S$ subject to the relations $v_sv_t=v_{st}$ and
\[v_s^{\phantom{*}}v_s^*v_t^{\phantom{*}}v_t^* = \begin{cases} v_r^{\phantom{*}}v_r^* &\text{ if } sS\cap tS= rS;\\
0 &\text{ if } s\perp t,
\end{cases}\] for $s,t\in S.$
It follows that $C^*(S)$ admits a dense spanning set $v_sv_t^*$ with $s,t\in S$, see \cites{Nor1,BLS2}. We let $e_{sS}$ denote the projection $v_sv_s^*$ in $C^*(S)$, for each $s\in S$.

The algebra $C^*(S)$ has a natural quotient, the \emph{boundary quotient} $\CQ(S)$ constructed in \cite{BRRW}: Let $\pi$ be the quotient map $C^*(S) \to \CQ(S)$ obtained by adding the relation $\prod_{s \in F} (1-v_s^{\phantom{*}}v_s^*)=0$ for every foundation set $F\ssubset S$ to the defining relations of $C^*(S)$.

We denote by $w_a, a \in S_c$ the standard generating isometries of $C^*(S_c)$, and by $\varphi\colon C^*(S_c) \to C^*(S)$ the $*$-homomorphism induced by the inclusion $S_c\subset S$. We are grateful to Sergey Neshveyev for the observation that groupoid techniques show the following result, cf. \cite{NS1}, which allows us to improve and simplify our construction of the core Fock module in subsection~\ref{subsec:construction}.

%\begin{remark}\label{rem:faithful varphi}
There exists a conditional expectation $E\colon C^*(S) \to C^*(S_c)$ given by
\begin{equation}\label{eq:faithful varphi}
v_s^{\phantom{*}}v_t^* \mapsto \begin{cases} \varphi^{-1}(v_s^{\phantom{*}}v_t^*)& \text{if }s,t\in S_c;\\
0&\text{otherwise}.\end{cases}
\end{equation} 
In particular, the $*$-homomorphism $\varphi\colon C^*(S_c) \to C^*(S)$ arising from the inclusion $S_c \subset S$ is faithful since $E\circ\varphi=\id$.
%\end{remark}

Let us record the following observation on the existence of normalised traces:

\begin{lemma}\label{lem:trace on C*(S)}
For a right LCM monoid $T$, the $C^*$-algebra $C^*(T)$ admits a normalised trace if and only if $T$ is left reversible, that is, $s\Cap t$ holds for all $s,t \in T$.
\end{lemma}
\begin{proof} Let $T$ be a right LCM monoid. Suppose first that there exist $s,t \in T$ with $s\perp t$. If $\tau$ were a normalised trace on $C^*(T)$, we would get
\[1 = \tau(1) \geq \tau(v_s^{\phantom{*}}v_s^*+v_t^{\phantom{*}}v_t^*) = \tau(v_s^{\phantom{*}}v_s^*)+ \tau(v_t^{\phantom{*}}v_t^*) = 2\tau(1)=2,\]
a contradiction.

Now suppose that $T$ is left reversible. Let $I_\ell(T)$ be the ($0$-bisimple) inverse monoid known as the left inverse hull of $T$, generated by the left translations $\lambda_t\colon T\to T, r\mapsto tr$ for $t\in T$. As observed in \cite[Lemma 3.29]{Nor1}, the inverse monoid $I_\ell(T)$ is without $0$ element. There is a maximal group homomorphic image of $I_\ell(T)$, denoted $G(T)$ in \cite[\S 3.4]{Nor1}, accompanied by a  quotient homomorphism $\alpha_T:I_\ell(T)\to G(T)$. At the level of $C^*$-algebras, there is a surjective $*$-homomorphism $\pi_T:C^*(I_\ell(T))\to C^*(G(T))$. Since $T$ is right LCM, then by following the proof of \cite[Proposition 3.27]{Nor1} we obtain a canonical isomorphism $\eta:C^*(T)\to C^*(I_\ell(T))$, $\eta(v_t)=\lambda_t$ for $t\in T$. We remark that  \cite[Proposition 3.27]{Nor1} is stated for an arbitrary right LCM monoid, and then $\eta$ maps onto $C_0^*(I_\ell(T))$, which is the quotient of $C^*(I_\ell(T))$ by the ideal corresponding to $0$ in $I_\ell(T)$. However, as we already pointed out, when $T$ is moreover left reversible then $0\notin I_\ell(T)$, and the map $\eta$ can be defined with image in $C^*(I_\ell(T))$.

By composing any normalised trace on $C^*(G(T))$ first with the $*$-homomorphism $\pi_T$, then the isomorphism $\eta$, we obtain a normalised trace on $C^*(T)$, as desired.
\end{proof}

\begin{proposition}\label{prop:trace on C*(S_c)}
The core $S_c$ is the largest submonoid $T$ of a right LCM monoid $S$ such that $C^*(T)$ admits a normalised trace. Moreover, $\tau_0$ given by $\tau_0(v_a^{\phantom{*}}v_b^*) = \delta_{a,b}$ for $a,b \in S_c$  defines a normalised trace on $C^*(S_c)$.
\end{proposition}
\begin{proof}
This follows immediately from Lemma~\ref{lem:trace on C*(S)} and its proof as the group $C^*$-algebra $C^*(G(S_c))$ admits the trivial trace which induces $\tau_0$ through $\pi_{S_c}$.
\end{proof}

\subsection{\texorpdfstring{KMS-states on $C^*$-algebras of right LCM monoids}{KMS-states on C*-algebras of right LCM monoids}} We briefly review the notions of KMS- and ground states for a given $C^*$-dynamical system $(A,\sigma)$ consisting of a $C^*$-algebra $A$ and a time evolution $\sigma:\R\curvearrowright A$, see \cite{BRII} for a standard introduction. An element $x \in A$ is  \emph{analytic} for $\sigma$ if the map $u \mapsto \sigma_u(x)$ with $u\in \R$ has a (unique) extension to an  analytic function $z \mapsto \sigma_z(x)$ from $\C$ into $A$. For $\beta>0$, a state $\phi$ of $A$ is a $\sigma$-KMS$_\beta$-state, or simply a KMS$_\beta$-state, if it satisfies the KMS$_\beta$ condition
\begin{equation}
\phi(xy) = \phi(y\sigma_{i\beta}(x))
\end{equation}
for all analytic $x,y \in A$.  To prove that a given $\sigma$-invariant state of $A$ is a KMS$_\beta$-state, it suffices to check the KMS$_\beta$-condition for all $x$ in a set of analytic elements that generate $A$ as a $C^*$-algebra and all $y$ in a  dense spanning subspace of $A$, see \cite[Lemma 1.9]{ALN}.

A \emph{KMS$_\infty$-state} of $A$ is a weak$^*$ limit of KMS$_{\beta_n}$-states as $\beta_n\to \infty$, see \cite{con-mar}. A state $\phi$ of $A$ is a \emph{ground state} of $A$ if $z\mapsto \phi(x\sigma_z(y))$ is bounded on the upper-half plane for all analytic $x,y\in A$.

Given a right LCM monoid $S$ with a generalised scale $N$, there is a time evolution $\sigma\colon \R \curvearrowright C^*(S)$ determined by $\sigma_u(v_s) := N_s^{iu} v_s$ for $s \in S, u \in \R$. Clearly, all isometries $v_s$, $s\in S$, are analytic.

Whenever $\phi$ is a $\kms_\beta$-state for $(C^*(S),\sigma)$, we let $(\pi_\phi,H_\phi, \xi_\phi)$ be the associated GNS representation and denote by $\tilde{\phi}$ its vector state extension to $\CL(\CH_\phi)$. Note that $\tilde{\phi}$ is a KMS$_\beta$-state on $\pi_\phi(C^*(S))$. For $I\subset \text{Irr}(N(S))$, we define a projection in $\pi_\phi(C^*(S))''$  by
\begin{equation}\label{eq:QI}
\begin{array}{c} Q_I:=\prod\limits_{n \in I}\bigl(1-\sum\limits_{[t] \in N^{-1}(n)/_\sim} \pi_\phi(e_{tS})\bigr). \end{array}
\end{equation}
These are well-defined by the following lemma, and will play a key role for the reconstruction formula for KMS-states established in Subsection~\ref{subsec:rec formula}.

\begin{lemma}\label{lem:core elts are unitaries under kms-states}
Let $S$ be a right LCM semigroup and suppose $\sigma$ is a time evolution on $C^*(S)$ such that $v_a$ is an analytic isometry in $C^*(S)$ with $\phi(v_a^{\phantom{*}}v_a^*)=1$ for all $a\in S_c$. If $\phi$ is a $\sigma$-$\kms_\beta$-state  for some $\beta\in \R$, then for all analytic elements $x,y \in C^*(S)$  and $a \in S_c$ we have $\phi(xe_{aS}y) = \phi(xy)$. In particular, we have $\pi_\phi(e_{sS}) = \pi_\phi(e_{tS})$  whenever $s \sim t$.
\end{lemma}
\begin{proof}
The assumption and the $\kms_\beta$-condition imply that
\[\phi(xe_{aS}y) = \phi(e_{aS}y\sigma_{i\beta}(x)e_{aS}) = \phi(y\sigma_{i\beta}(x)) = \phi(xy).\]
Now let $s\sim t$ in $S$. Thus $sa=tb$ for some $a,b \in S_c$, hence the previous step implies
\begin{align*}
\bigl(\pi_\phi(e_{sS}y)\xi_\phi\mid\pi_\phi(x)\xi_\phi\bigr)
&= \phi(x^*e_{sS}y) = \phi(x^*e_{saS}y) \\
&= \phi(x^*e_{tbS}y) = \bigl( \pi_\phi(e_{tS}y)\xi_\phi\mid\pi_\phi(x)\xi_\phi\bigr)
\end{align*}
for all analytic elements $x,y \in C^*(S)$. Since the linear span of analytic elements is dense in $C^*(S)$, we conclude that $\pi_\phi(e_{sS})=\pi_\phi(e_{tS})$.
\end{proof}

\section{Main results on KMS states}\label{sec:main result}

In this section we summarise our results on KMS-states in the form of two theorems. Our first result is Theorem~\ref{thm:KMS results}, which gives a complete characterisation of the $\kms_\beta$-state structure of $C^*(S)$ for $\beta$ outside the critical interval $[1,\beta_c]$,  and provides sufficient conditions under which we get unique $\kms_\beta$-state at each $\beta$ inside the critical interval. To prove Theorem~\ref{thm:KMS results} we follow the strategy of the proof of \cite{ABLS}*{Theorem~4.3}, which deals with the more specialised case of admissible semigroups. The arguments of \cite{ABLS} that we cannot use --- namely, those involving the factorisation properties of admissible monoids --- have been amended, indeed sharpened,  and are presented in Sections~\ref{sec:ingredients}, \ref{sec:param of KMS-states} and ~\ref{sec:uniqueness}.

To state the first result we need to introduce some notation. We keep the explanation for the following definitions at a minimum here, and refer to Section \ref{sec:uniqueness} for a thorough discussion. Throughout this section $S$ will denote a right LCM monoid with generalised scale $N$. For $a,b \in S_c$ with $a\neq b$ and $n \in N(S)$ we define
 \begin{align}
    F_n^{a,b} &  := \{ [s] \in N^{-1}(n)/_\sim \mid asc=bsd \text{ for some } c,d \in S_c\};\label{eq:F-n}\text{ and }\\
     A_n^{a,b} &  := \{ [s] \in N^{-1}(n)/_\sim \mid asc=bsc \text{ for some } c \in S_c\}.\label{eq:A-n}
 \end{align}
 We refer to the equivalence classes in $A_n^{a,b}$ as \emph{absorbing} elements.

Note that if $S$ is right cancellative, then these sets are empty. However, if $S$ is not right cancellative then $ A_n^{a,b}$ can be nonempty even though $S_c$ is right cancellative, for example a group (as in the case of self-similar actions).

We now introduce our criterion for uniqueness of KMS$_1$-states.
\begin{definition}\label{def:neg q-a}
Let $S$ be a right LCM monoid  with generalised scale $N$. For $a,b \in S_c$  we say that $S$ is  \emph{$(a,b)$-regular} if
\begin{equation}\label{eq:ab-qa}
\frac{\lvert F_n^{a,b}\setminus A_n^{a,b}\rvert}{n} \stackrel{n \to \infty}{\longrightarrow} 0.
\end{equation}
$S$ is called \emph{core regular} if it is $(a,b)$-regular for all $a,b \in S_c$.
\end{definition}

Our second criterion is modelled on systems with a critical interval of the form $(1,\beta_c]$.

\begin{definition}\label{def:sum q-a} % such that $1<\beta_c$
Let $S$ be a right LCM monoid $S$ with generalised scale $N$. For $\beta \in (1,\beta_c]$ and $a,b \in S_c$ we say that $S$ is \emph{$\beta$-summably $(a,b)$-regular} if
\begin{equation}\label{eq:ab-summable-qa}
\begin{array}{c}\lim\limits_{I}\zeta_I(\beta)^{-1} \sum\limits_{n \in \langle I\rangle^+} n^{-\beta} \lvert F_n^{a,b} \setminus A_n^{a,b} \rvert = 0,\end{array}
\end{equation}
as $I\ssubset \text{Irr}(N(S))$ increases to $\text{Irr}(N(S))$. We say $S$ is \emph{$\beta$-summably core regular} if it is $\beta$-summably $(a,b)$-regular for all $a,b\in S_c$, and \emph{summably core regular} if it is $\beta$-summably core regular for all $\beta \in (1,\beta_c]$.
\end{definition}

\begin{thm}\label{thm:KMS results}
Suppose $S$ is a right LCM monoid with generalised scale $N$. For the time evolution $\sigma$ on $C^*(S)$ determined by $N$, we have:
\begin{enumerate}[(1)]
\item There are no $\kms_\beta$-states on $C^*(S)$ for $\beta < 1$.
\item For $\beta > \beta_c$, there is an affine homeomorphism between $\kms_\beta$-states on $C^*(S)$ and normalised traces on $C^*(S_c)$.
\item There is an affine homeomorphism between ground states on $C^*(S)$ and states on $C^*(S_c)$. In case that $\beta_c < \infty$, a ground state is a $\kms_\infty$-state if and only if it corresponds to a normalised trace on $C^*(S_c)$.
\item There is a $\kms_1$-state $\psi_1$ for $(C^*(S),\sigma)$  determined by
\[\psi_1(v_a^{\phantom{*}}v_b^*) = \lim\limits_{n \in N(S)} \frac{\lvert A_n^{a,b}\rvert}{n} \quad \text{for } a,b \in S_c.\]
If $S$ is core regular, then $\psi_1$ is the unique $\kms_1$-state.
\item There is a $\kms_\beta$-state $\psi_\beta$ for $(C^*(S),\sigma)$ for each  $\beta \in (1,\beta_c]$ determined by
\[\begin{array}{c}
\psi_\beta(v_a^{\phantom{*}}v_b^*) = \lim\limits_{I \ssubset \text{Irr}(N(S))} \zeta_I(\beta)^{-1} \sum\limits_{n \in \langle I\rangle^+} n^{-\beta} \lvert A_n^{a,b}\rvert \quad \text{for } a,b \in S_c.
\end{array}\]
If $S$ is $\beta$-summably core regular, then $\psi_\beta$ is the unique  KMS$_\beta$-state.
\end{enumerate}
\end{thm}
It follows in particular that when $S$ is core regular and summably core regular, there is a unique $\kms_\beta$-state for each $\beta\in [1,\beta_c]$.

\begin{proof}[Proof of Theorem~\ref{thm:KMS results}] Part (1) is verbatim as in \cite[Theorem 4.3]{ABLS}. Statement (2) is the content of Proposition~\ref{prop:KMS-states parametrization above crit}. Part (3) combines Proposition~\ref{prop:construction of ground states} with the argument of the similar statement in \cite[Theorem 4.3]{ABLS}. Part (4) is proved in Proposition~\ref{prop:uniqueness at 1 if neg q-a}, and (5) follows from  Proposition~\ref{prop:uniqueness in the crit interval by s q-a}.
\end{proof}

\begin{remark}\label{rem:main thm extended}
Recall that $\pi\colon C^*(S) \to \CQ(S)$ denotes the $*$-epimorphism to the boundary quotient of $S$. Since $\sigma$ descends to a time evolution on $\CQ(S)$, we have that a $\kms_\beta$-state for $(C^*(S),\sigma)$ factors through $\pi$ precisely when $\beta=1$. Moreover, no ground state on $C^*(S)$ factors through $\pi$. The proofs of these two claims are the same as in \cite{ABLS}*{Theorem~4.3(5) and (8)} upon replacing $\pi_p$ in (5) with $\pi$.
\end{remark}

For every normalised trace $\tau$ on $C^*(S_c)$ and $\beta>1$, Proposition~\ref{prop:construction of KMS-states} will produce a KMS$_\beta$-state $\psi_{\beta,\tau}$ given by
\[\begin{array}{c} \psi_{\beta,\tau} = \lim\limits_{I \nearrow \text{Irr}(N(S))} \zeta_I(\beta)^{-1} \sum\limits_{n \in \langle I\rangle^+} n^{1-\beta}\psi_{\tau,n} \end{array}\]
with $\psi_{\tau,n}(x) = n^{-1}\sum_{[s] \in N^{-1}(n)/_\sim} \tau\circ E(v_s^*xv_s^{\phantom{*}})$ for $x \in C^*(S)$, where $E\colon C^*(S) \to C^*(S_c)$ is the conditional expectation from \eqref{eq:faithful varphi}. By Proposition~\ref{prop:alg char of KMS-states}, $\psi_{\beta,\tau}\circ \varphi$ is again a normalised trace on $C^*(S_c)$. We will be interested in the fixed points under the self-maps $\chi_\beta$ of the simplex of normalised traces on $C^*(S_c)$ given by $\tau\mapsto \psi_{\beta,\tau}\circ\varphi$ for $\beta \in (1,\infty)$.

\begin{definition}\label{def:trace simplex funct eq}
For $\beta_0 \in [1,\infty)$, let $T_{>\beta_0}(C^*(S_c))=\bigcap_{\beta \in (\beta_0,\infty)} T(C^*(S_c))^{\chi_\beta}$ denote the set of normalised traces $\tau$ on $C^*(S_c)$ satisfying
\begin{equation}\label{eq:KMS funct eq}
\chi_\beta(\tau) = \psi_{\beta,\tau}\circ \varphi = \tau \quad \text{for all }\beta>\beta_0.
\end{equation}
\end{definition}

We note that $T_{>\beta_0}(C^*(S_c))$ is a simplex.

\begin{thm}\label{thm:restricting to tau}
Suppose $S$ is a right LCM monoid with generalised scale $N$, and consider the time evolution $\sigma$ on $C^*(S)$ determined by $N$. For every $\beta \in [1,\infty)$, the map $\tau\mapsto \psi_{\beta,\tau}$ is an embedding of the simplex $T_{>\beta}(C^*(S_c))$ into the simplex of $\kms_\beta$-states. For $\beta=1$, this map is an affine homeomorphism between $T_{>1}(C^*(S_c))$ and the $\kms_1$-states.
\end{thm}

Theorem~\ref{thm:restricting to tau} is an entirely new result that explores the theory in a complementary direction to the findings of \cite{CSi}. Its proof is presented at the end of Section~\ref{sec:param of KMS-states}.

\section{The toolkit trilogy: algebraic constraints, a reconstruction formula, and the core Fock module}\label{sec:ingredients}

We assume throughout this section that $S$ is a right LCM monoid with a generalised scale $N$ and we let $\sigma$ be the associated time evolution on $C^*(S)$.

\subsection{Algebraic constraints}\label{subsec:algebraic constraints}
The results on algebraic constraints for $\kms$-states from \cite{ABLS}*{Section~6} apply to $(C^*(S),\sigma)$ with the exception \cite{ABLS}*{Proposition~6.4}, which we shall now improve, and some minor parts of \cite{ABLS}*{Propositions~6.6 and 6.7}, which we will not need. For completeness, we note that $\kms_0$-states are the traces on $C^*(S)$, but under our assumption $C^*(S)$ is traceless. Indeed, $N$ is assumed nontrivial, so we can consider a transversal $\CT_n$ for $N^{-1}(n)/_\sim$, where $n \in N(S)\setminus\{1\}$, and by the definition of $N$, $\CT_n$ contains $n$ mutually orthogonal elements. Thus  a trace $\tau$ on $C^*(S)$ would satisfy $1= \tau(1)\geq \tau(\sum_{t \in \CT_n} e_{tS}) = n$, which is impossible. 

\begin{proposition}\label{prop:alg char of KMS-states}
Let $\beta \in \R\setminus\{0\}$. A state $\phi$ on $C^*(S)$ is a $\kms_\beta$-state if and only if $\phi\circ \varphi$ is a normalised trace on $C^*(S_c)$ and
\begin{equation}\label{eq:alg char of KMS-states}
\phi(v_s^{\phantom{*}}v_t^*) = \begin{cases} N_s^{-\beta} \phi(v_b^{\phantom{*}}v_a^*) & \text{ if } sa=tb \text{ for some } a,b \in S_c;\\
0 & \text{ otherwise}\end{cases}
\end{equation}
holds for all  $s,t \in S$.
\end{proposition}
\begin{proof}
Suppose that $\phi$ is a $\kms_\beta$-state. Since $N(a)=1$ for all $a\in S_c$, see \cite{ABLS}*{Proposition~3.6(i)}, the KMS-condition implies that $\phi\circ\varphi$ is tracial. For $s,t \in S$, applying the $\kms_\beta$-condition twice gives
\[\phi(v_s^{\phantom{*}}v_t^*) = N_s^{-\beta}\phi(v_t^*v_s^{\phantom{*}}) = (N_t/N_s)^\beta \phi(v_s^{\phantom{*}}v_t^*).\]
Thus, if $N_s\neq N_t$, necessarily $\phi(v_s^{\phantom{*}}v_t^*)=0$. If $N_s=N_t$, then \cite{ABLS}*{Proposition~3.6(ii)} states that either $s\perp t$, in which case $\phi(v_s^{\phantom{*}}v_t^*) = N_s^{-\beta}\phi(v_t^*v_s^{\phantom{*}}) =0$, or $s\sim t$. In the later case, assume $sS\cap tS = saS, sa=tb$ for some $a,b \in S_c$. Then  $\phi(v_s^{\phantom{*}}v_t^*) = N_s^{-\beta}\phi(v_t^*v_s^{\phantom{*}}) = N_s^{-\beta} \phi(v_b^{\phantom{*}}v_a^*)$. Suppose $c,d \in S_c$ also satisfy $sc=td$. Since $sa$ is a right LCM of $s$ and $t$, there is $f \in S_c$ such that $sc=saf, td=tbf$. Left cancellation  implies $c=af$ and $d=bf$, and by Lemma~\ref{lem:core elts are unitaries under kms-states} we obtain  $\phi(v_d^{\phantom{*}}v_c^*) = \phi(v_b^{\phantom{*}}e_{fS}^{\phantom{*}}v_a^*) = \phi(v_b^{\phantom{*}}v_a^*)$, showing \eqref{eq:alg char of KMS-states}.

Conversely, suppose $\phi$ is a state such that $\phi\circ\varphi$ is a trace and \eqref{eq:alg char of KMS-states} holds. Note that $\phi$ is $\sigma$-invariant because $N_a=1$ for all $a\in S_c$.
In view of \cite{ALN}*{Lemma~1.9}, it suffices to establish that
\begin{equation}\label{eq:KMS-condition on spanning family halfway}
\phi(v_s^{\phantom{*}}v_t^{\phantom{*}}v_r^*) = N_s^{-\beta} \phi(v_t^{\phantom{*}}v_r^*v_s^{\phantom{*}}) \quad \text{for all } s,t,r \in S,
\end{equation}
By \eqref{eq:alg char of KMS-states}, $\phi(v_s^{\phantom{*}}v_t^{\phantom{*}}v_r^*)=0$  unless $st\sim r$, and $v_r^*v_s^{\phantom{*}} = 0$ unless $r \Cap s$. Assuming $rS \cap sS = ss'S, ss'=rr'$ for some $s',r' \in S$, we get $v_t^{\phantom{*}}v_r^*v_s^{\phantom{*}} = v_{tr'}^{\phantom{*}}v_{s'}^*$, so again by \eqref{eq:alg char of KMS-states} $\phi(v_t^{\phantom{*}}v_r^*v_s^{\phantom{*}}) = 0$ unless $tr'\sim s'$. Assume therefore that $tr'a=s'b$ for some  $a,b \in S_c$. This leads to $str'a = ss'b = rr'b$, hence to $N_{st} = N_r$ because $N$ is multiplicative on $S$ and takes value $1$ on $S_c$. It follows from \cite{ABLS}*{Proposition~3.6(ii)} that $st \sim r$ as well. So both sides of \eqref{eq:KMS-condition on spanning family halfway} are nonzero simultaneously, and it remains to show that they are equal in this case.

By Lemma~\ref{lem:consequence-ABLS-N-homomorphism},  $r' \in S_c$. Using that $\phi\circ \varphi$ is a trace on $C^*(S_c)$, we get
\[\phi(v_t^{\phantom{*}}v_r^*v_s^{\phantom{*}}) = \phi(v_{tr'}^{\phantom{*}}v_{s'}^*) \stackrel{\eqref{eq:alg char of KMS-states}}{=}  N_t^{-\beta} \phi(v_b^{\phantom{*}}v_a^*) = N_t^{-\beta} \phi(v_{r'b}^{\phantom{*}}v_{r'a}^*) \stackrel{\eqref{eq:alg char of KMS-states}}{=} N_s^\beta \phi(v_{st}^{\phantom{*}}v_r^*),\]
giving the desired claim.
\end{proof}

\subsection{The reconstruction formula}\label{subsec:rec formula}
The aim of this section is to obtain the reconstruction formula for $\kms_\beta$-states of  \cite{ABLS}*{Lemma~7.5} for all right LCM semigroups $S$ with generalised scale $N$. The result  was first proved in \cite{LR2}*{Lemma~10.1} for $S=\mathbb{N}\rtimes \mathbb{N}^\times$, and generalised to an admissible right LCM monoid $S$  in \cite{ABLS}*{Lemma~7.5}. Here we show that the existence of a generalised scale alone suffices to obtain the formula. While the strategy of proof is the same as for \cite{ABLS}*{Lemma~7.5}, the major difference is that the former result assumed the existence of minimal representatives for the equivalence classes in $S/_\sim$, see \cite{ABLS}*{Lemma~3.2}, which allows one to perform most steps in $C^*(S)$, whereas here we need to work almost entirely in $\pi_\phi(C^*(S))$, for instance because the summations appearing in \eqref{eq:QI} would not be well-defined when replacing $\pi_\phi(e_{tS})$ by $e_{tS}$.

\begin{lemma}\label{lem:rec formula}
Let $\phi$ be a $\kms_\beta$-state on $C^*(S)$ for some $\beta\in \R$, let $(\pi_\phi,\CH_\phi,\xi_\phi)$  its GNS-representation and $\tilde{\phi} = \bigl( \cdot \ \xi_\phi\mid \xi_\phi \bigr)$ the vector state extension of $\phi$ to $\CL(\CH_\phi)$. For every nonempty subset $I$ of $\Irr(N(S))$, $Q_I := \lim\limits_{F\ssubset I}Q_F$ defines a projection in $\pi_\phi(C^*(S))''$. If $\zeta_I(\beta) < \infty$, then the following statements hold:
\begin{enumerate}[(i)]
\item $\tilde{\phi}(Q_I) = \zeta_I(\beta)^{-1}$.
\item The map $y \mapsto \zeta_I(\beta) \ \tilde{\phi}(Q_I\pi_\phi(y)Q_I)$ defines a state $\phi_I$ on $C^*(S)$, for which $\phi_I \circ \varphi$ is a trace on $C^*(S_c)$.
\item The family $(\pi_\phi(v_s^{\phantom{*}})Q_I\pi_\phi(v_s^*))_{[s] \in S_I/_\sim}$ consists of mutually orthogonal projections, and
%\[\begin{array}{c}
$Q^I := \sum_{[s] \in S_I/_\sim} \pi_\phi(v_s^{\phantom{*}})Q_I^{\phantom{*}}\pi_\phi(v_s^*)$
%\end{array}\]
defines a projection such that $\tilde{\phi}(Q^I) = 1$.
\item There is a reconstruction formula for $\phi$ given by
\begin{equation}\label{eq:rec formula}
\begin{array}{c}
\phi(y) =  \zeta_I(\beta)^{-1} \sum\limits_{[s] \in S_I/_\sim} N_s^{-\beta} \ \phi_I(v_s^*yv_s^{\phantom{*}}) \quad \text{for all } y\in C^*(S).
\end{array}
\end{equation}
\end{enumerate}
\end{lemma}
\begin{proof}
The strategy of proof is similar to \cite{ABLS}*{Lemma~7.5}, so we only indicate what changes can be made to avoid using core irreducible elements (whose collection in $S$ may reduce to $1$, see Proposition~\ref{prop:BS no core irreds}).

Note that $(Q_F)_{F\ssubset I}$ with $Q_F$ as in \eqref{eq:QI} is a family of commuting projections such that $F\subset F'$ implies $Q_F \geq Q_{F'}$. Thus they converge weakly to a projection $\lim_{F\ssubset I} Q_F =: Q_I$. To obtain (i), we note that
\begin{equation}\label{eq:defect proj prod as sum}\begin{array}{c}
Q_F =  \sum\limits_{A\subset F} (-1)^{\lvert A\rvert}\sum\limits_{([t_n])_{n \in A} \in \prod\limits_{n \in A} N^{-1}(n)/_\sim}\hspace*{-5mm}\pi_\phi(e_{\bigcap\limits_{n \in A}t_nS}) = \sum\limits_{A\subset F} (-1)^{\lvert A\rvert} \sum\limits_{[t] \in N^{-1}(m_A)/_\sim}\hspace*{-5mm}\pi_\phi(e_{tS})
\end{array}\end{equation}
 for $F\ssubset I$. Then the argument employed in \cite{ABLS}*{Lemma~7.5} gives that
  $\tilde{\phi}(Q_F) = \zeta_F(\beta)^{-1}$, so the claim (i) follows by continuity.
This implies that $\phi_I$ is a state on $C^*(S)$. To prove (ii) it suffices, since $\phi_I$ is the w*-limit of $(\phi_F)_{F\ssubset I}$, to show that $\phi_F\circ\varphi$ is tracial for an arbitrary, but fixed $F\ssubset I$. We claim that
\[\begin{array}{lclcl}
\phi_F(x) &=& \zeta_I(\beta)\sum\limits_{A,B\subset F} (-1)^{\lvert A\rvert+\lvert B\rvert} \sum\limits_{[r] \in N^{-1}(m_{A\cup B})/_\sim} N(r)^{-\beta}\phi(v_r^*xv_r^{\phantom{*}}),
\end{array}\]
for all $x \in \varphi(C^*(S_c))$.
To prove this claim, note that by \eqref{eq:defect proj prod as sum}, $\phi_F(x)$ will contain summands of the form $\phi(e_{sS}xe_{tS})$ with $[s]\in N^{-1}(m_{A})$, $[t]\in N^{-1}(m_{B})$, which by the $\kms$-condition vanish unless $s\Cap t$. For any pair $[s],[t]$ with $sS\cap tS=rS$, we get summands $\phi(xe_{rS})$, which by Lemma~\ref{lem:core elts are unitaries under kms-states} only depend on the equivalence class $[r]$ in $S/_\sim$. Further, since $N$ preserves right LCMs by \cite{ABLS}*{Proposition 3.6},  the right LCM of $m_A$ and $m_B$ in $N(S)$ is given by $m_{A\cup B}$. Finally, the claim follows by applying the $\kms_\beta$-condition to $\phi(xe_{rS})$.

Next we recall that the bijection $\alpha_a\colon S/_\sim \to S/_\sim$ restricts to a bijection on $N^{-1}(n)/_\sim$ for $a \in S_c$, see \cite{ABLS}*{Lemma~3.9} (the proof given there only used the existence of a generalised scale on $S$, although the result was stated with further assumptions on $S$). Using this,  for arbitrary $a,b,c,d \in S_c$ and $n \in N(S)$ we have
\[\begin{array}{c}
\sum\limits_{[q] \in N^{-1}(n)/_\sim}\hspace*{-3mm}\phi(v_q^*v_a^{\phantom{*}}v_b^*v_c^{\phantom{*}}v_d^*v_q^{\phantom{*}}) \stackrel{\alpha_a}{=} \sum\limits_{[q] \in N^{-1}(n)/_\sim}\hspace*{-3mm}\phi(v_{bq}^*v_c^{\phantom{*}}v_d^*v_a^{\phantom{*}}v_b^*v_{bq}^{\phantom{*}})
\stackrel{\alpha_b}{=} \sum\limits_{[q] \in N^{-1}(n)/_\sim}\hspace*{-3mm}\phi(v_q^*v_c^{\phantom{*}}v_d^*v_a^{\phantom{*}}v_b^*v_q^{\phantom{*}}).
\end{array}\]
From this observation and the above claim it follows that $\phi_F\circ \varphi$ is a normalised trace on $C^*(S_c)$ for every $F\ssubset I$, and hence $\phi_I\circ \varphi$ is also tracial.

To prove (iii), fix projections $\pi_\phi(v_s^{\phantom{*}})Q_I\pi_\phi(v_s^*)$ and $\pi_\phi(v_t^{\phantom{*}})Q_I\pi_\phi(v_t^*)$ for $[s]\neq [t]$ in $S_I/_\sim$. Then $[s]\in N^{-1}(n)/_\sim$ and $[t]\in N^{-1}(m)/_\sim$ for some $m,n\in N(S)$. Moreover, $s\not\sim t$. If $m=n$ we get $v_s^*v_t=0$ exactly as in the proof of \cite{ABLS}*{Lemma~7.5}. Suppose that $m\neq n$. If $s\perp t$ then again $v_s^*v_t=0$, so we may assume $sS\cap tS=ss'S, ss'=tt'$ for some $s',t'\in S$, and thus $v_s^*v_t^{\phantom{*}}=v_{s'}^{\phantom{*}}v_{t'}^*$. Moreover,
since $s\not\sim t$, we have $N_{s'}>1$ or $N_{t'}>1$. Assume $N_{s'}>1$ and pick a divisor $n'\in \text{Irr}(N(S))$ of $N_{s'}$ in $N(S)$. Then $Q_I\pi_\phi(v_s^*v_t)Q_I$ contains a factor
\[
\begin{array}{c}\bigl(1-\sum\limits_{[r]\in N^{-1}(n')/_\sim}\pi_\phi(e_{rS})\bigr)\pi_\phi(e_{s'S}),\end{array}
\]
and there is at least one $r \in N^{-1}(n')$ satisfying $s'S\subset rS$. The latter yields $\pi_\phi(e_{rS}e_{s'S}) = \pi_\phi(e_{s'S})$, so the whole expression $Q_I\pi_\phi(v_s^*v_t)Q_I$ vanishes. The case of $N_{t'}>1$ is analogous and orthogonality of the family follows. For $I$ finite, we can now conclude that
\[\begin{array}{c}
\tilde{\phi}(Q^I) = \sum\limits_{[s] \in S_I/_\sim} \tilde{\phi}(\pi_\phi(v_s)Q_I\pi_\phi(v_s)^*) = \sum\limits_{[s] \in S_I/_\sim} N_s^{-\beta} \tilde{\phi}(Q_I) \stackrel{(i)}{=} 1
\end{array}\]
using $Q_I \in \pi_\phi(C^*(S))$ and the KMS-condition for $\tilde{\phi}$. Part (iii) is now completed for arbitrary $I$ by invoking continuity.

Claim (iv) follows from (iii) together with the KMS-condition for $\tilde{\phi}$: For $I$ finite and $y \in C^*(S)$, we first get
\[\begin{array}{lcl}
\phi(y) &=& \sum\limits_{[s],[t] \in S_I/_\sim} \tilde{\phi}(\pi_\phi(v_s)Q_I\pi_\phi(v_s^*yv_t^{\phantom{*}})Q_I\pi_\phi(v_t)^*) \\
&=& \sum\limits_{[s],[t] \in S_I/_\sim} N_s^{-\beta} \tilde{\phi}(Q_I\pi_\phi(v_s^*yv_t^{\phantom{*}})Q_I\pi_\phi(v_t)^*\pi_\phi(v_s)Q_I).
\end{array}\]
According to (iii), the summands vanish unless $[s]=[t]$, which gives $\pi_\phi(v_s)Q_I\pi_\phi(v_s)^* = \pi_\phi(v_t)Q_I\pi_\phi(v_t)^*$. Thus we can let $t=s$ so that
\[\begin{array}{lcl}
\phi(y) &=& \sum\limits_{[s] \in S_I/_\sim} N_s^{-\beta} \tilde{\phi}(Q_I\pi_\phi(v_s^*yv_t^{\phantom{*}})Q_I) \\
&\stackrel{(ii)}{=}& \zeta_I(\beta)^{-1}\sum\limits_{[s],[t] \in S_I/_\sim} N_s^{-\beta} \phi_I(v_s^*yv_t^{\phantom{*}}).
\end{array}\]
The case of arbitrary $I$ is obtained by a continuity argument.
\end{proof}

\begin{remark}\label{rem:reconstruction for beta=1}
Lemma~\ref{lem:rec formula} bears no implications for $\beta=1$ as $\zeta_I(1)=\infty$ for all non-empty infinite $I\subset \text{Irr}(N(S))$. However, we can obtain a simpler reconstruction formula, as follows. If $\phi$ is a $\kms_1$-state on $C^*(S)$ and $n\in N(S)$, then
\[\begin{array}{c}
\tilde{\phi}(1-Q_{\{n\}}) = \sum\limits_{[s] \in N^{-1}(n)/_\sim} \phi(e_{sS}) =1.
\end{array}\]
As in the proof of Lemma~\ref{lem:rec formula}~(iv) we arrive at
%\[
%\phi(y)=\tilde{\phi}\bigl(\sum\limits_{[s],[t]\in N^{-1}(n)/_\sim}\pi_\phi(e_{sS})\pi_\phi(y)\pi_\phi(e_{tS})\bigr)=\sum\limits_{[s]\in N^{-1}(n)/_\sim}\phi(ye_{sS}),
%\]
%for $y\in C^*(S)$.
%where we employed the KMS condition to move the projections $e_{sS}$ across $y$ and we invoked \cite{ABLS}*{Proposition 3.6(ii)} to conclude that the only nontrivial contributions to the double sum above occur when $s\sim t$. By applying the KMS$_1$ condition again, we thus get
\begin{equation}\label{eq:rec formula beta=1}
\begin{array}{c} \phi(y) = \sum\limits_{[s] \in N^{-1}(n)/_\sim} n^{-1}\phi(v_s^*yv_s^{\phantom{*}}) \quad \text{for all } y \in C^*(S), n \in N(S). \end{array}
\end{equation}
%By \cite{ABLS}*{Proposition~3.6(vi)}, the boundary quotient $\CQ(S)$ of $C^*(S)$ is the quotient by the relations $\sum_{[s] \in N^{-1}(n)/_\sim} \overline{e_{sS}} = 1$ for all $n \in N(S)$. In fact, it suffices to require this for all $n \in \text{Irr}(N(S))$. It is implicit in these relations that $\overline{e_{sS}}=\overline{e_{tS}}$ whenever $s\sim t$.
\end{remark}

\begin{remark}\label{rem:ALN-paper}
There is an apparent similarity between this type of reconstruction formula and \cite{ALN}*{Theorem~6.8}. We need Lemma~\ref{lem:rec formula} in the style of \cite{ABLS} as we will make use of the intermediate results we get for finite subsets in the case where $1 <\beta \leq \beta_c$. This will be the key to proving Proposition~\ref{prop:uniqueness in the crit interval by s q-a}, which is the  uniqueness result for KMS-states of infinite type in the case of very high-dimensional dynamics, that is, where the critical inverse temperature is strictly larger than $1$.
\end{remark}

\subsection{The core Fock module}\label{subsec:construction}
The aim of this subsection is to introduce a $C^*(S)$-$C^*(S_c)$-module for every right LCM monoid $S$ that can be employed to induce KMS-states on $C^*(S)$ from normalised traces on $C^*(S_c)$, and ground states from states.  

For each $s \in S$, we let $\FM_{0,s}$ denote a copy of $C^*(S_c)$ equipped with the standard right $C^*(S_c)$-module structure. We write $\varphi_s\colon \FM_{0,s} \to C^*(S_c)$ for the natural identification map, which is an isometric isomorphism of Banach spaces, and set $\varphi_{s,t}:=\varphi_s^{-1}\circ \varphi_t^{\phantom{1}}\colon \FM_{0,t} \to \FM_{0,s}$. For convenience, let $\tilde{\varphi}_t := \varphi\circ \varphi_t: \FM_{0,t} \to C^*(S)$.

\begin{lemma}\label{lem:inner product}
On the right $C^*(S_c)$-module $\FM_0 := \bigoplus_{s \in S} \FM_{0,s}$, the map
\[\begin{array}{rcl}
\langle \cdot, \cdot \rangle\colon \FM_{0,s} \times \FM_{0,t} &\to& C^*(S_c) \\
(\xi, \eta) &\mapsto& E((v_s\tilde{\varphi}_s(\xi))^*v_t\tilde{\varphi}_t(\eta))
\end{array}\]
determines a positive semidefinite sesquilinear form on $\FM_0$.
\end{lemma}
\begin{proof}
It is straightforward to check that $\langle \cdot, \cdot \rangle$ is sesquilinear, so let $\xi= \sum_{s \in F} \xi_s$ with $\xi_s \in \FM_{0,s}$ for every $t \in F\ssubset S$. Then $\varphi(\langle \xi, \xi\rangle) = \tilde{\xi}^*\tilde{\xi} \geq 0$ in $C^*(S)$ for $\tilde{\xi}:= \sum_{s \in F} v_s\varphi_s(\xi_s)$. Since $\varphi$ is faithful, we deduce $\langle \xi, \xi\rangle \geq 0$.
\end{proof}

We can thus form the right Hilbert $C^*(S_c)$-module $\FM := \overline{\FM_0/ \FN}^{\|\cdot\|}$ for $\FN := \{ \xi \in \FM_0 \mid \langle\xi,\xi\rangle=0\}$ and the norm given by $\|\xi\|^2:= \|\langle \xi,\xi\rangle\|$.

%\begin{remark}\label{rem:inner product decomposition}
It follows from the definition of $\langle\cdot,\cdot\rangle$ that for $\xi= \sum_{t \in S} \xi_t, \eta=\sum_{r \in S} \eta_r$ in $\FM_0$, if we let $\xi_{[s]} := \sum_{t \in [s]} \xi_t$ and similarly for $\eta_{[s]}$, then
\begin{equation}\label{eq:inner product decomposition}
\begin{array}{c}
\langle \xi,\eta\rangle = \sum\limits_{[s] \in S/_\sim} \sum\limits_{t,r \in [s]}\langle \xi_t,\eta_r\rangle = \sum\limits_{[s] \in S/_\sim} \langle \xi_{[s]},\eta_{[s]}\rangle,
\end{array}
\end{equation}
 because all terms $\langle \xi_t,\eta_r\rangle$ vanish for $[t]\neq [r]$. From this perspective, it is natural to view $\FM_0$ as $\FM_0= \bigoplus_{[s] \in S/_\sim} \FM_{0,[s]}$ with $\FM_{0,[s]} := \bigoplus_{t \in [s]} \FM_{0,t}$.
%\end{remark}

\begin{thm}\label{thm:core Fock module}
For every right LCM monoid $S$, $\FM$ is a $C^*(S)-C^*(S_c)$ right Hilbert bimodule with left action determined by $V_s([\xi]) := [\varphi_{st,t}(\xi)]$ for $\xi \in \FM_{0,t}$.
\end{thm}
\begin{proof}
We need to show that the linear maps $(V_s)_{s \in S}$ are well-defined and adjointable on $\FM_0/\FN$ so that they extend to adjointable, linear maps on $\FM$ that we again denote by $V_s$, and that these define a representation of $C^*(S)$. First let $V_{0,s}$ be the linear operator on $\FM$ given by $V_{0,s} (\xi_t) := \varphi_{st,t}(\xi_t)$ for $\xi_t \in \FM_{0,t}$ and $s,t \in S$. For $r \in S$ and $\xi \in \FM_{0,t},\eta \in \FM_{0,r}$, we have
\[\begin{array}{lclcl}
\langle V_{0,s}(\xi), V_{0,s}(\eta)\rangle &=& E(\tilde{\varphi}_{st}(\varphi_{st,t}(\xi))^*v_{st}^*v_{sr}^{\phantom{*}}\tilde{\varphi}_{sr}(\varphi_{sr,r}(\eta))) \vspace*{2mm}\\
&=& E(\tilde{\varphi}_t(\xi)^*v_t^*v_r^{\phantom{*}}\tilde{\varphi}_r(\eta))\\
&=& \langle \xi,\eta\rangle,
\end{array}\]
because $[st]=[sr]$ precisely when $ [t]=[r]$ due to left cancellation. Thus $V_{0,s}$ induces a well-defined inner-product preserving linear operator $V_s$  on $\FM_0/\FN$. In particular, if $V_s$ is shown to be adjointable, it will be an isometry.

Fix $s \in S$. For each $t \in S$ such that $[t] \in sS/_\sim$, we fix $a_t \in S_c, \overline{t}\in S$ with the property that $sS\cap tS = ta_tS$ and  $ta_t=s\overline{t}$. We then define a linear operator $V_{0,s}'$ on $\FM_0$ by sending $\xi_t \in \FM_{0,t}$ to
\begin{equation}\label{eq:adjoint of V_S}%\footnote{This can be done in $C^*(S_c)$ or $C^*(S)$.}
\begin{array}{c}
V_{0,s}' (\xi_t) = \begin{cases} \tilde{\varphi}_{\overline{t}}^{-1}(v_{a_t}^*\tilde{\varphi}_t(\xi_t)) &\text{ if } sS\cap tS = ta_tS, ta_t=s\overline{t} \text{ with } a_t \in S_c, \overline{t}\in S;\\
0 &\text{ otherwise.}\end{cases}
\end{array}
\end{equation}
We claim that $\langle V_{0,s}'(\xi),\eta\rangle = \langle \xi,V_{0,s}(\eta)\rangle$ holds for all $\xi,\eta \in \FM_0$. Due to the definition of $\langle\cdot,\cdot\rangle$ and linearity, it suffices to consider the case where $\xi \in \FM_{0,t}$ and $\eta \in \FM_{0,r}$ for arbitrary $t,r \in S$. We have $\langle V_{0,s}'(\xi),\eta\rangle = 0$ unless $[t] \in sS/_\sim$ and $\eta \in \FM_{0,r}$ with $r\sim \overline{t}$, that is, by left cancellation in $S$, unless  $t\sim s\overline{t}\sim sr$. By \eqref{eq:inner product decomposition}, we see that $\langle \xi,V_{0,s}(\eta)\rangle=0$ unless $sr\sim t$. Assuming $t\sim sr$, we compute that
%and let $a,b \in S_c$ satisfy $tS\cap srS=taS, ta=srb$. Then
\[\begin{array}{lclclcl}
\langle V_{0,s}' (\xi), \eta \rangle
%&=& \langle \tilde{\varphi}_{\overline{t}}^{-1}(v_{a_t}^*\tilde{\varphi}_t(\xi)), \eta\rangle
&=& E(\tilde{\varphi}_t(\xi)^*v_{a_t}^{\phantom{*}}v_{\overline{t}}^*v_r^{\phantom{*}}\tilde{\varphi}_r(\eta)) \vspace*{2mm}\\
&=& E(\tilde{\varphi}_t(\xi)^*v_t^*v_{sr}^{\phantom{*}}\tilde{\varphi}_r(\eta))
&=&  \langle \xi,V_{0,s}(\eta)\rangle.
\end{array}\]
It follows that $\langle V_{0,s}'(\xi),\eta\rangle = \langle \xi,V_{0,s}(\eta)\rangle$ holds for all $\xi,\eta \in \FM_0$. In particular, $\FN \subset \ker~V_{0,s}'$ for all $s \in S$ and $\|V_{0,s}'\|\leq 1$, so that we obtain a well-defined bounded linear operator $V_s'$ on $\FM$. We conclude that $V_s$ is an adjointable isometry on $\FM$ with $V_s'=V_s^*$ for every $s \in S$, and $V_sV_t = V_{st}$ for $s,t \in S$ is clear.

We note for later use that the expression for $V_{0,s}' (\xi_t)$ in \eqref{eq:adjoint of V_S} depends on the chosen pair $a_t$ and $\overline{t}$. We can replace this by any other pair $f\in S_c$ and $w\in S$ such that $rS\cap sS=sfS$ and $tf=sw$ upon multiplying $a_t$ and $\overline{t}$ on the right with an element $x$ in $S^*$. The resulting $V_{0,s}' (\xi)$ will land in a summand $\FM_{0,\overline{t}x}$, but the final operator $V_s^*$ will not, as it is defined on the module $\FM$.

In order to obtain a representation of $C^*(S)$, we need to show that
\begin{equation}\label{eq:C*(S) key relation}
V_s^*V_t^{\phantom{*}}=
\begin{cases}
V_{s'}^{\phantom{*}}V_{t'}^* & \text{ if } sS\cap tS=ss'S, ss'=tt'; \\
0 &\text{ if } s\perp t.
\end{cases}
\end{equation}
holds for all $s,t \in S$. So let $s,t \in S$. If $s\perp t$, then  $s \perp tr$ for all $r \in S$, and therefore $V_{0,t}(\FM_{0,r}) = \FM_{0,tr}$, implying $V_s^*V_t^{\phantom{*}}=0$. So let us assume that there are $s',t' \in S$ with $sS\cap tS=ss'S, ss'=tt'$. Fix $r \in S$ and $\xi \in \FM_{0,r}$, and note that
\begin{equation}\label{eq:C*(S) key relation ideals}
sS\cap trS = sS\cap tS\cap trS = t(t'S\cap rS).
\end{equation}
By our definition of adjoint in \eqref{eq:adjoint of V_S}, we have $V_s^*V_t^{\phantom{*}}([\xi]) = 0$ unless $[tr]\in sS/_\sim$, and
similarly $V_{s'}^{\phantom{*}}V_{t'}^*([\xi])$ vanishes unless $[r]\in t'S/_\sim$. In case the former inclusion holds, we have $trb=ss''$ for some $b\in S_c$, $s''\in S$, thus $trb\in sS\cap tS=tt'S$, and hence by left cancellation $rb=t'r'$ for some $r'\in S$. This is exactly the condition $[r]\in t'S/_\sim$. Conversely, from $[r]\in t'S/_\sim$ we get $[tr]\in sS/_\sim$ by \eqref{eq:C*(S) key relation ideals}. Thus it suffices to prove equality of the terms in \eqref{eq:C*(S) key relation} evaluated at $[\xi]$ under the assumption that they both are non-zero. We choose
 $a_r \in S_c$ and $ \overline{r} \in S$ with $t'S\cap rS = ra_rS, ra_r=t'\overline{r}$, so that
 \[
 V_{s'}^{\phantom{*}}V_{t'}^*([\xi])
= [\varphi_{s'\overline{r},\overline{r}}(\tilde{\varphi}_{\overline{r}}^{-1}(v_{a_r}^*\tilde{\varphi}_r(\xi)))].
 \]
 Note that by left cancellation for $t$ and \eqref{eq:C*(S) key relation ideals}, we equivalently have  $sS\cap trS= tra_rS$ with $tra_r = tt'\overline{r}= ss'\overline{r}$. Similarly, choose
  $a_{tr}\in S_c$ and $\overline{tr}\in S$ such that $trS\cap sS=tra_{tr}S$ and $tra_{tr}=s\overline{tr}$ to obtain
\[
V_s^*V_t^{\phantom{*}}([\xi]) = [\tilde{\varphi}_{\overline{tr}}^{-1}(v_{a_{tr}}^*\tilde{\varphi}_{tr}(\varphi_{tr,r}(\xi)))].
\]
Since we may compute $V_s^*$ using the pair $a_r$, $s'\overline{r}$ or $a_{tr}$, $\overline{tr}$, we therefore get
\[\begin{array}{lclcl}
V_{s'}^{\phantom{*}}V_{t'}^*([\xi])
&=& [\tilde{\varphi}_{s'\overline{r}}^{-1}(v_{a_r}^*\tilde{\varphi}_r(\xi))] \vspace*{2mm}\\
&=& [\tilde{\varphi}_{\overline{tr}}^{-1}(v_{a_{tr}}^*\tilde{\varphi}_{tr}(\varphi_{tr,r}(\xi)))]
&=& V_s^*V_t^{\phantom{*}}([\xi])
\end{array}\]
By continuity, we get \eqref{eq:C*(S) key relation}, so that $V$ defines a left action of $C^*(S)$ on $\FM$.
\end{proof}

\begin{definition}\label{def:core Fock module}
For a right LCM monoid $S$, we let the \emph{core Fock module} of $S$ be the right Hilbert $C^*(S)$-$C^*(S_c)$-module $\FM$ of Theorem~\ref{thm:core Fock module}. We let $\pi_\FM$ denote the $*$-homomorphism $C^*(S) \to \CL(\FM)$ from Theorem~\ref{thm:core Fock module}.
\end{definition}

By virtue of Theorem~\ref{thm:core Fock module}, every representation $\pi\colon C^*(S_c) \to B(H)$ of $C^*(S_c)$ on a Hilbert space $H$ gives rise to an induced representation $\text{Ind}^\FM \pi\colon C^*(S) \to \CL(\FM \otimes_\pi H)$ of $C^*(S)$. In particular, this applies to the GNS-representation $(\pi_\rho,H_\rho,\xi_\rho)$ associated to a state $\rho$ on $C^*(S_c)$.

\begin{notation}\label{not:unit vectors}
For $r \in S$, let $\omega_r:= [\varphi_r^{-1}(1)]$. For every state $\rho$ on $C^*(S_c)$, define a state on $C^*(S)$ by
%\[\begin{array}{c}
$\chi_{\rho,r}(x) := \langle \text{Ind}^\FM\pi_\rho(x)(\omega_r \otimes \xi_\rho),\omega_r \otimes \xi_\rho \rangle$ for $x\in C^*(S)$.
%\end{array}\]
\end{notation}

\begin{lemma}\label{lem:evaluation under state}
Let $r \in S$ and $\rho$ be a state on $C^*(S_c)$. Then the following  hold:
\begin{enumerate}[(i)]
\item $\chi_{\rho,r}(x) = \rho(E(v_r^*xv_r^{\phantom{*}}))$ for all $r \in S, x \in C^*(S)$.
%$\chi_{\rho,r}(v_s^{\phantom{*}}v_t^*) = \rho(E(v_r^*v_s^{\phantom{*}}v_t^*v_r^{\phantom{*}}))$ holds for $r,s,t \in S$.
In particular, $\chi_{\rho,1}\circ \varphi = \rho$.
\item If $\rho$ is a trace, then $\chi_{\rho,r}=\chi_{\rho,r'}$ whenever $r,r' \in S$ satisfy $r\sim r'$.
\item If $\rho$ is a trace and $s,t \in S$, then $\chi_{\rho,r}(v_s^{\phantom{*}}v_t^*)=0$, unless $s\sim t$, say $sS\cap tS=saS, sa=tb$ for some $a,b \in S_c$, and there exists $r''\in S$ with $sar''=tbr''\sim r$. The map $[r]\mapsto [r'']$ is a one-to-one correspondence between $saS/_\sim$ and $S/_\sim$ with $\chi_{\rho,r}(v_s^{\phantom{*}}v_t^*)= \chi_{\rho,r''}(v_a^*v_b^{\phantom{*}})$.
\end{enumerate}
\end{lemma}
\begin{proof}
We first show that $\langle V_s^{\phantom{*}}V_t^*(\omega_r),\omega_r\rangle = E(v_r^*v_t^{\phantom{*}}v_s^*v_r^{\phantom{*}})$ holds for all $s,t,r \in S$, which will then imply (i) as $C^*(S) = \overline{\text{span}}\{ v_s^{\phantom{*}}v_t^* \mid s,t \in S\}$. Due to \eqref{eq:adjoint of V_S}, $V_t^*(\omega_r)=0$ unless $r\in tS/_\sim$, in which case writing $rS\cap tS = ra_rS, ra_r=t\overline{r}$ for some $a_r \in S_c $ and $ \bar{r} \in S$ implies that $V_t^*(\omega_r)\in \FM_{0,\bar{r}}$. Thus $\langle V_s^{\phantom{*}}V_t^*(\omega_r),\omega_r\rangle=0$
 unless we have $r\in tS/_\sim$ and $r\sim s\bar{r}$. Assume therefore $r\in tS/_\sim$ and $r\sim s\bar{r}$, with $a_r, \bar{r}$ defined as above.  It follows as in the proof of \eqref{eq:C*(S) key relation} that
 \[
 \langle V_s^{\phantom{*}}V_t^*(\omega_r),\omega_r\rangle = E(v_a^{\phantom{*}}v_{s\overline{r}}^*v_r^{\phantom{*}}) = E(v_r^*v_t^{\phantom{*}}v_s^*v_r^{\phantom{*}}).
 \]
Likewise, $v_r^*v_t^{\phantom{*}}v_s^*v_r^{\phantom{*}}$ is zero, unless $rS\cap tS = tt'S, rr'=tt'$ for some $r',t' \in S$. In this case, we have $v_r^*v_t^{\phantom{*}}v_s^*v_r^{\phantom{*}}= v_{r'}^{\phantom{*}}v_{st'}^*v_r^{\phantom{*}}$, which belongs to $\varphi(C^*(S_c))$ if and only if $st' \sim r$.

To prove part (ii) note that $\rho(w_a^{\phantom{*}}w_a^*) =1$ for all $a \in S_c$ because $\rho$ is a trace and $w_a$ is an isometry, thus (i) implies $\chi_{\rho,ra}= \chi_{\rho,r}$  for all $r \in S$. As $s\sim t$ is equivalent to $sa=tb$ for some $a,b \in S_c$, we get (ii).
%As $\langle V_s^{\phantom{*}}V_t^*(\omega_r),\omega_r\rangle = \langle V_t^*(\omega_r),V_s^*(\omega_r)\rangle$, the expression vanishes unless $[r] \in (sS\cap tS)/_\sim$, see \eqref{eq:adjoint of V_S}. For $[r] \in (sS\cap tS)/_\sim$, the images $V_t^*(\omega_r)$ and $V_s^*(\omega_r)$ lie in $\FM_{0,r_t}+\FN$ and $\FM_{0,r_s}+\FN$, respectively. Therefore, the inner product vanishes unless $[r_t]=[r_s]$, in which case we get the claimed formula by the definition of $\langle\cdot,\cdot\rangle$. The formula involving $\varphi\colon C^*(S_c) \to C^*(S)$ is a direct consequence.

For (iii), let $s,t\in S$. Part (i) implies that $\chi_{\rho,r}(v_s^{\phantom{*}}v_t^*)$ vanishes unless $sr' \sim r \sim tr'$ for some $r' \in S$. If these equivalences hold, then $s \sim t$ by \cite{ABLS}*{Proposition~3.6(ii)} and $N_sN_{r'} = N_r = N_tN_{r'}$. So fix $a,b \in S_c$ satisfying $sS\cap tS=saS $ and $ sa=tb$. Moreover, $sr'\sim tr'$ gives $sr'c=tr'd \in sS\cap tS = saS$ for suitable $c,d \in S_c$, so there is $[r''] \in S/_\sim$ with $[r] = [sr'] = [sar'']$ and $sar'' = tbr''$. Conversely, every $[r''] \in S$ yields a distinct class $[sar'']$. We thus obtain a one-to-one correspondence $[r] \mapsto [r'']$ between $saS/_\sim = tbS/_\sim$ and $S/_\sim$. Using this correspondence, we get
\[\chi_{\rho,r}(v_s^{\phantom{*}}v_t^*) \stackrel{(ii)}{=} \chi_{\rho,sar''}(v_s^{\phantom{*}}v_t^*) \stackrel{(i)}{=} \rho(E(v_{sar''}^*v_s^{\phantom{*}}v_t^*v_{tbr''}^{\phantom{*}}))
\stackrel{(i)}{=} \chi_{\rho,r''}(v_a^*v_b^{\phantom{*}}).\]
\end{proof}

%\begin{remark}\label{rem:evaluation against 1}
%We record two implications of Lemma~\ref{lem:evaluation against vectors} for a state $\rho$ on $C^*(S_c)$:
%\end{remark}
\section{Parametrisation of KMS-states}\label{sec:param of KMS-states}
\subsection{Ground states and KMS-states of finite type}\label{subsec:GS and finite type}
With the core Fock module in place, we proceed with the discussion of KMS-states for $(C^*(S),\sigma)$. We first address ground states in Proposition~\ref{prop:construction of ground states}, then we produce KMS$_\beta$-states from normalised traces on $C^*(S_c)$ in Proposition~\ref{prop:construction of KMS-states} for $\beta \in (1,\infty)$ and obtain a parametrisation for the $\kms_\beta$-states with  $\beta \in (\beta_c,\infty)$ in Proposition~\ref{prop:KMS-states parametrization above crit}. In the spirit of \cite[Definition 6.4]{ALN}, we refer to these last ones as states of finite type.

\begin{proposition}\label{prop:construction of ground states}
There exists an affine homeomorphism between the states on $C^*(S_c)$ and the ground states on $C^*(S)$ given by $\rho \mapsto \psi_\rho := \chi_{\rho,1}$.
\end{proposition}
\begin{proof}
Given a state $\rho$ on $C^*(S_c)$, the map $\psi_\rho$ is a state on $C^*(S)$ such that $0 \neq \psi_\rho (v_s^{\phantom{*}}v_t^*)$ forces $[1] \in (sS\cap tS)/_\sim$, see Lemma~\ref{lem:evaluation under state}~(i). As this condition is equivalent to $s,t \in S_c$, the proof then follows as in \cite{ABLS}*{Proposition~6.2}.
\end{proof}

Before we construct $\kms_\beta$-states on $C^*(S)$ from traces $\tau$ on $C^*(S_c)$ by use of the GNS representation $\pi_\tau$, we make note of an intermediate step which produces states on $C^*(S)$ from traces on $C^*(S_c)$. Recall that we view the finite subsets of $\text{Irr}(N(S))$ as a (countable) directed set when ordered by inclusion.

\begin{lemma}\label{lem:construction of KMS-states}
Let $\tau$ be a normalised trace on $C^*(S_c)$. For every $n \in N(S)$, the state $\psi_{\tau,n} := n^{-1}\sum_{[s] \in N^{-1}(n)/_\sim} \chi_{\tau,s}$ on $C^*(S)$ restricts to a normalised trace $\psi_{\tau,n}\circ \varphi$ on $C^*(S_c)$. In particular, for all $\beta >1$ and $I \subset \text{Irr}(N(S))$ with $\zeta_I(\beta) < \infty$,
\[\begin{array}{c}
\psi_{\beta,\tau,I} := \zeta_I(\beta)^{-1} \sum\limits_{[s] \in S_I/_\sim} N_s^{-\beta} \chi_{\tau,s} = \zeta_I(\beta)^{-1} \sum\limits_{n \in \langle I\rangle^+} n^{1-\beta}\psi_{\tau,n}
\end{array}\]
defines a state on $C^*(S)$ such that $\psi_{\beta,\tau,I} \circ \varphi$ is tracial on $C^*(S_c)$.
\end{lemma}
\begin{proof}
By Lemma~\ref{lem:evaluation under state}~(ii) and the trace property for $\tau$, the map $\psi_{\tau,n}$ is a well-defined state on $C^*(S)$. We claim that
\begin{equation}\label{eq:trace on level n}
\begin{array}{c}
\sum\limits_{[s] \in N^{-1}(n)/_\sim} \chi_{\tau,s}(v_a^{\phantom{*}}v_b^*v_c^{\phantom{*}}v_d^*) = \sum\limits_{[s] \in N^{-1}(n)/_\sim} \chi_{\tau,s}(v_c^{\phantom{*}}v_d^*v_a^{\phantom{*}}v_b^*)
\end{array}
\end{equation}
holds for all $a,b,c,d \in S_c$ and $n \in N(S)$. The term $\chi_{\tau,s}(v_a^{\phantom{*}}v_b^*v_c^{\phantom{*}}v_d^*)$ vanishes unless $[s] \in \text{Fix}~\alpha_a^{\phantom{1}}\alpha_b^{-1}\alpha_c^{\phantom{1}}\alpha_d^{-1}$, see Lemma~\ref{lem:evaluation under state}~(i). Likewise, $\chi_{\tau,s'}(v_c^{\phantom{*}}v_d^*v_a^{\phantom{*}}v_b^*)$ is zero unless $[s'] \in \text{Fix}~\alpha_c^{\phantom{1}}\alpha_d^{-1}\alpha_a^{\phantom{1}}\alpha_b^{-1}$. These two sets are related by the bijections
\begin{equation}\label{eq:permutation makes sets meet}
\text{Fix}~\alpha_a^{\phantom{1}}\alpha_b^{-1}\alpha_c^{\phantom{1}}\alpha_d^{-1} \stackrel{\alpha_d}{\longleftarrow} \text{Fix}~\alpha_d^{-1}\alpha_a^{\phantom{1}}\alpha_b^{-1}\alpha_c^{\phantom{1}} \stackrel{\alpha_c}{\longrightarrow} \text{Fix}~\alpha_c^{\phantom{1}}\alpha_d^{-1}\alpha_a^{\phantom{1}}\alpha_b^{-1}.
\end{equation}
We note that every $s \in S$ satisfies %For $[s] \in \text{Fix}~\alpha_d^{-1}\alpha_a^{\phantom{1}}\alpha_b^{-1}\alpha_c^{\phantom{1}}$ we find that
\begin{equation}\label{eq:permutation makes tau values meet}
\begin{array}{lclcl}
\chi_{\tau,ds}(v_a^{\phantom{*}}v_b^*v_c^{\phantom{*}}v_d^*) &=& \tau(\varphi^{-1}(v_{ds}^*v_a^{\phantom{*}}v_b^*v_c^{\phantom{*}}v_d^*v_{ds}^{\phantom{*}})) &=& \tau(\varphi^{-1}(v_{ds}^*v_a^{\phantom{*}}v_b^*v_{cs}^{\phantom{*}})) \\
&=&\tau(\varphi^{-1}(v_{cs}^*v_c^{\phantom{*}}v_d^*v_a^{\phantom{*}}v_b^*v_{cs}^{\phantom{*}})) &=& \chi_{\tau,cs}(v_c^{\phantom{*}}v_d^*v_a^{\phantom{*}}v_b^*).
\end{array}
\end{equation}
%and likewise
%\[\chi_{\tau,cs}(v_c^{\phantom{*}}v_d^*v_a^{\phantom{*}}v_b^*) = \tau(\varphi^{-1}(v_{cs}^*v_c^{\phantom{*}}v_d^*v_a^{\phantom{*}}v_b^*v_{cs}^{\phantom{*}})) = \tau(\varphi^{-1}(v_{ds}^*v_a^{\phantom{*}}v_b^*v_{cs}^{\phantom{*}})).\]
As $\alpha$ restricts to an action of $S_c$ by bijections on $N^{-1}(n)/_\sim$, Lemma~\ref{lem:evaluation under state}~(ii) gives
\[\{\chi_{\tau,es} \mid [s] \in N^{-1}(n)/_\sim \} = \{\chi_{\tau,s} \mid [s] \in N^{-1}(n)/_\sim \} \quad \text{for all } n \in N(S), e\in S_c.\]
This leads to
\[\begin{array}{lcl}
\sum\limits_{[s] \in N^{-1}(n)/_\sim} \chi_{\tau,s}(v_a^{\phantom{*}}v_b^*v_c^{\phantom{*}}v_d^*) &=& \sum\limits_{\substack{[s] \in N^{-1}(n)/_\sim, \\ [s] \in \text{Fix}~\alpha_a^{\phantom{1}}\alpha_b^{-1}\alpha_c^{\phantom{1}}\alpha_d^{-1}}} \chi_{\tau,s}(v_a^{\phantom{*}}v_b^*v_c^{\phantom{*}}v_d^*) \vspace*{2mm}\\
&=& \sum\limits_{\substack{[s] \in N^{-1}(n)/_\sim, \\ [s] \in \text{Fix}~\alpha_d^{-1}\alpha_a^{\phantom{1}}\alpha_b^{-1}\alpha_c^{\phantom{1}}}} \chi_{\tau,ds}(v_a^{\phantom{*}}v_b^*v_c^{\phantom{*}}v_d^*) \vspace*{2mm}\\
&\stackrel{\eqref{eq:permutation makes sets meet},\eqref{eq:permutation makes tau values meet}}{=}& \sum\limits_{\substack{[s] \in N^{-1}(n)/_\sim, \\ [s] \in \text{Fix}~\alpha_d^{-1}\alpha_a^{\phantom{1}}\alpha_b^{-1}\alpha_c^{\phantom{1}}}} \chi_{\tau,cs}(v_c^{\phantom{*}}v_d^*v_a^{\phantom{*}}v_b^*) \vspace*{2mm}\\
&=& \sum\limits_{[s] \in N^{-1}(n)/_\sim} \chi_{\tau,s}(v_c^{\phantom{*}}v_d^*v_a^{\phantom{*}}v_b^*),
\end{array}\]
which establishes \eqref{eq:trace on level n}. Hence $\psi_{\tau,n} \circ \varphi$ is a trace on $C^*(S_c)$. The claim concerning $\psi_{\beta,\tau,I}$ is an immediate consequence hereof.
\end{proof}

\begin{proposition}\label{prop:construction of KMS1-states}
For every normalised trace $\tau$ on $C^*(S_c)$, $\psi_{1,\tau} := \lim\limits_{n \in N(S)} \psi_{\tau,n}$ defines a $\kms_1$-state on $C^*(S)$.
\end{proposition}
\begin{proof}
If the limit exists, it is necessarily a state whose restriction to $\varphi(C^*(S_c))$ is tracial. It thus suffices to show that \eqref{eq:alg char of KMS-states} holds for any given $s,t \in S$ for $\psi_{\tau,n}$ for all large enough $n \in N(S)$. According to Lemma~\ref{lem:evaluation under state}~(iii), $\chi_{\tau,r}(v_s^{\phantom{*}}v_t^*)$ vanishes for all $r$ unless $s\sim t$, so let us assume that $sS\cap tS=saS, sa=tb$ for some $a,b \in S_c$. Let $n$ be large enough so that $N_s=N_t$ divides $n$ in $N(S)$, say $n=N_sm$ for some $m\in N(S)$. Next, we note that the one-to-one correspondence $[r]\mapsto [r'']$ from $saS/_\sim$ to $S/_\sim$ from Lemma~\ref{lem:evaluation under state}~(iii) restricts to a bijection $\{[r] \in N^{-1}(n)/_\sim \mid s\Cap r\} \to N^{-1}(m)/_\sim$. We infer that
\[\begin{array}{c}
\psi_{\tau,n}(v_s^{\phantom{*}}v_t^*) = n^{-1}\sum\limits_{[r] \in (N^{-1}(n)\cap saS)/_\sim} \chi_{\tau,r}(v_s^{\phantom{*}}v_t^*) \stackrel{\eqref{lem:evaluation under state}(iii)}{=} N_s^{-1}\psi_{\tau,m}(v_a^*v_b^{\phantom{*}}).
\end{array}\]
Since $\psi_{\tau,m}\circ \varphi$ is tracial and $n\nearrow \infty$ is the same as $m \nearrow \infty$, this establishes \eqref{eq:alg char of KMS-states} for $\psi_{1,\tau}$. The limit exists by weak*-compactness: we first get a limit for some sequence $(n_k)_{k \in \N} \subset N(S)$ with $n_k\to \infty$, and then by using \eqref{eq:rec formula beta=1} we get that this limit does not depend on the choice of  sequence.
\end{proof}

\begin{proposition}\label{prop:construction of KMS-states}
For $\beta > 1$ and every normalised trace $\tau$ on $C^*(S_c)$, there is a sequence $(I_k)_{k \geq 1}$ of finite subsets of $\text{Irr}(N(S))$ such that $\psi_{\beta,\tau,I_k}$ weak*-converges to a $\kms_\beta$-state $\psi_{\beta,\tau}$ on $C^*(S)$ as $I_k \nearrow \text{Irr}(N(S))$. For $\beta \in (\beta_c,\infty)$, the state $\psi_{\beta,\tau}$ is given by $\psi_{\beta,\tau,\text{Irr}(N(S))}$.
\end{proposition}
\begin{proof}
Let $\beta>1$ and $\tau$ be a normalised trace on $C^*(S_c)$. Due to weak*-compactness of the state space on $C^*(S)$, there is a sequence $(I_k)_{k \geq 1}$ of finite subsets of $\text{Irr}(N(S))$ with $I_k \nearrow \text{Irr}(N(S))$ such that $(\psi_{\beta,\tau,I_k})_{k \geq 1}$ obtained from Lemma~\ref{lem:construction of KMS-states} converges to some state $\psi_{\beta,\tau}$ in the weak* topology. By Proposition~\ref{prop:alg char of KMS-states}, $\psi_{\beta,\tau}$ is a $\kms_\beta$-state if and only if $\psi_{\beta,\tau}\circ \varphi$ defines a trace on $C^*(S_c)$ and \eqref{eq:alg char of KMS-states} holds. Note that $\psi_{\beta,\tau} \circ \varphi$ is tracial on $C^*(S_c)$ because each of the $\psi_{\beta,\tau,I_k}$ has this property by Lemma~\ref{lem:construction of KMS-states}.

It remains to prove \eqref{eq:alg char of KMS-states}. Let $s,t\in S$.
% that is, $\psi_{\beta,\tau}(v_s^{\phantom{*}}v_t^*) = \delta_{[s],[t]} \ N_s^{-\beta} \psi_{\beta,\tau}(v_a^*v_b^{\phantom{*}})$ for $s,t \in S$, where $sS\cap tS = saS, sa=tb$ for some $a,b \in S_c$ given that $[s] = [t]$.
For $k\geq 1$, Lemma~\ref{lem:evaluation under state}~(iii) shows that $\chi_{\tau,r}(v_s^{\phantom{*}}v_t^*)$ vanishes for $[r] \in S_{I_k}/_\sim$ unless $s\sim t$ and $sr' \sim r \sim tr'$ for some $r' \in S$. In particular, $k$ must be so large that $N_s=N_t \in \langle I_k \rangle^+$, which we assume from now on. Hence $\psi_{\beta,\tau}(v_s^{\phantom{*}}v_t^*) = 0$ unless $[s]=[t]$, so fix $a,b \in S_c$ satisfying $sS\cap tS=saS $ and $ sa=tb$. Noting that the one-to-one correspondence $[r]\mapsto [r'']$ from Lemma~\ref{lem:evaluation under state}~(iii) maps $saS_{I_k}/_\sim$ to $S_{I_k}/_\sim$ as $r\sim sar''$ and $N_r,N_s \in \langle I_k\rangle^+$ force $N_{r''} \in \langle I_k\rangle^+$, we conclude that the equality $\psi_{\beta,\tau,I_k}(v_s^{\phantom{*}}v_t^*) = \delta_{[s],[t]} \ N_s^{-\beta} \psi_{\beta,\tau,I_k}(v_a^*v_b^{\phantom{*}})$ holds for $k$ large enough. Thus the limit $\psi_{\beta,\tau}$ satisfies \eqref{eq:alg char of KMS-states} for all $s,t \in S$, and hence is a $\kms_\beta$-state. The claim for $\beta \in (\beta_c,\infty)$ follows immediately from $I_k \nearrow \text{Irr}(N(S))$ because the formula from Lemma~\ref{lem:construction of KMS-states} makes sense for $I = \text{Irr}(N(S))$.
\end{proof}

%\begin{corollary}\label{rem:KMS_1-states as limits}
%The system $(C^*(S), \sigma)$ admits a $\kms_1$-state.
%\end{corollary}

%\begin{proof}
%Let $\tau$ be a normalised trace $\tau$ on $C^*(S_c)$, see Proposition~\ref{prop:trace on C*(S_c)}. By Proposition~\ref{prop:construction of KMS-states} we may pick a sequence $(\psi_{\beta_k,\tau})_{k \geq 1}$ of $\kms_{\beta_k}$-states for a sequence $(\beta_k)_{k\geq 1} \subset (1,\infty)$ with $\beta_k \searrow 1$. By weak* compactness, a subsequence of $(\psi_{\beta_k,\tau})_{k \geq 1}$ converges to a state $\psi_{1,\tau}$ on $C^*(S)$, and it follows from \cite{BRII}*{Proposition~5.3.23} that $\psi_{1,\tau}$ is a $\kms_1$-state.
%\end{proof}

\begin{proposition}\label{prop:KMS-states parametrization above crit}
Let $\beta \in (\beta_c,\infty)$. Then $\phi \mapsto \phi_{\text{Irr}(N(S))} \circ \varphi$ defines an affine homeomorphism between the $\kms_\beta$-states on $C^*(S)$ and the normalised traces on $C^*(S_c)$. Its inverse is given by $\tau \mapsto \psi_{\beta,\tau}$.
\end{proposition}
\begin{proof}
We shall prove the following:
\begin{enumerate}[(i)]
\item If $\phi$ is a $\kms_\beta$-state on $C^*(S)$, then $\psi_{\beta,\phi_{\Irr(N(S))} \circ \varphi} = \phi$.
\item If $\psi_{\beta, \tau}$ denotes the $\kms_\beta$-state obtained in Proposition~\ref{prop:construction of KMS-states} for $\text{Irr}(N(S))$ and a trace $\tau$ on $C^*(S_c)$, then $(\psi_{\beta,\tau})_{\text{Irr}(N(S))} \circ \varphi = \tau$.
\end{enumerate}
We start with (i): As $\phi_{\Irr(N(S))} \circ \varphi$ is a normalised trace on $C^*(S_c)$ by Lemma~\ref{lem:rec formula}~(ii), $\psi_{\beta,\phi_{\Irr(N(S))} \circ \varphi}$ is a $\kms_\beta$-state, see Proposition~\ref{prop:construction of KMS-states}. For $x \in C^*(S))$, we get
\[\begin{array}{c}
\psi_{\beta,\phi_{\Irr(N(S))} \circ \varphi}(x) \stackrel{\ref{lem:evaluation under state}(i)}{=} \zeta_S(\beta)^{-1} \sum\limits_{s \in S/_\sim} N_s^{-\beta} \phi_{\Irr(N(S))}(v_s^*xv_s^{\phantom{*}}) \stackrel{\ref{lem:rec formula}~(iv)}{=} \phi(x).
\end{array}\]

For (ii), we fix $x \in C^*(S_c)$ and get
\begin{equation}\label{eq:parametrization is surj}
\begin{array}{lcl}
(\psi_{\beta,\tau})_{\text{Irr}(N(S))}(\varphi(x)) &=& \zeta_S(\beta) \tilde{\psi}_{\beta,\tau}(Q_{\text{Irr}(N(S))}\pi_{\psi_{\beta,\tau}}(\varphi(x)) Q_{\text{Irr}(N(S))}) \vspace*{2mm}\\
&=& \lim\limits_{\substack{I \nearrow \text{Irr}(N(S))\\ \lvert I \rvert < \infty}} \zeta_S(\beta) \tilde{\psi}_{\beta,\tau}(Q_{I}\pi_{\psi_{\beta,\tau}}(\varphi(x)) Q_{I}).
\end{array}
\end{equation}
Let $(T_n)_{n \in I}$ be a family of transversals for $(N^{-1}(n)/_\sim)_{n \in I}$ (which we fix for the remainder of this proof). Note that by \eqref{eq:QI}, the defect projection $q_{I} = \prod_{n \in I}(1-\sum_{t \in T_n} e_{tS})$ is a preimage of $Q_I$ under $\pi_{\psi_{\beta,\tau}}$, so that $\tilde{\psi}_{\beta,\tau}(Q_I\pi_{\psi_{\beta,\tau}}(\varphi(x)) Q_I) = \psi_{\beta,\tau}(q_I\varphi(x)) q_I)$. By definition of $\psi_{\beta,\tau}$, we have
\begin{equation}\label{eq:param surj sum of chis}
\begin{array}{lcl}
\zeta_S(\beta) \psi_{\beta,\tau}(q_{I}\varphi(x) q_{I}) &=& \sum\limits_{[s] \in S/_\sim} N_s^{-\beta} \chi_{\tau,s}(q_{I}\varphi(x)q_{I}).
\end{array}
\end{equation}
We claim that $\chi_{\tau,s}(q_{I}\varphi(x)q_{I}) = \delta_{[s], S_c} \tau(x)$ for $I$ large enough and all $s$.
First, let $s \in S_c$. Since $\tau$ is a trace, Lemma~\ref{lem:evaluation under state}~(ii) yields $\chi_{\tau,s}(q_I\varphi(x)q_I) = \chi_{\tau,1}(q_I\varphi(x)q_I) = \tau(E(q_I\varphi(x)q_I))$, where $E$ is the conditional expectation from \eqref{eq:faithful varphi} given by $E(v_t^{\phantom{*}}v_r^*) = \chi_{S_c}(t)\chi_{S_c}(r) \ w_t^{\phantom{*}}w_r^*$. The term $q_I\varphi(x)q_I$ is of form $\varphi(x)+y$, where $y$ corresponds to the collection of cross-terms. Thus $y$ is a finite sum of elements of the form $\pm e_{tS}\varphi(x)e_{rS}$ with $t \notin S_c$ or $r \notin S_c$. But such terms belong to $\varphi(C^*(S_c))$ if and only if they vanish, so we get $\chi_{\tau,s}(q_I\varphi(x)q_I) = \tau(E(\varphi(x))=\tau(x)$.

Now suppose $s \in S\setminus S_c$, that is, $N_s >1$. For $I$ large enough, there is $n \in I$ with $N_s \in nN(S)$. Since $T_n$ is a foundation set, there is $t \in T_n$ such that $s \Cap t$, hence by Lemma~\ref{lem:consequence-ABLS-N-homomorphism} we can write $sS\cap tS = saS$ and $sa=tu$ for some $u\in S$ and $a \in S_c$. We claim that $s\perp t'$ for all $t' \in T_n$ such that $ t\neq t'$. If this were not true, then from $s\Cap t'$ for  $t'\in T_n$ with $t\neq t'$ we get as above $sS\cap t'S=sbS$ and $sb=t'r$ for some $r\in S$ and $b\in S_c$. Now pick $c,d\in S$ such that $ac=bd$ and note that $tuc=sac=sbd=t'rd$, which implies $t\Cap t'$ and contradicts the fact that $T_n$ is accurate.
 It follows that $v_s^*q_I\varphi(x)q_Iv_s$ begins with the factor $v_s^*q_{\{n\}}$, which by the claim is $v_s^*(1-e_{tS})$. From the choice of $a$, it is easy to see that $v_s^*(1-e_{tS})=(1-e_{aS})v_s^*$. By Lemma~\ref{lem:evaluation under state}~(ii) we write $\chi_{\tau,s}(q_I\varphi(x)q_I)=\tau(E(v_s^*q_{\{n\}}q_{I\setminus\{n\}}\varphi(x)q_Iv_s^{\phantom{*}}))$. Now denoting $z=q_{I\setminus\{n\}}\varphi(x)q_Iv_s^{\phantom{*}}$, we obtain
\[
\chi_{\tau,s}(q_I\varphi(x)q_I)=\tau(E((1-e_{aS})v_s^*z))=\tau((1-w_a^{\phantom{*}}w_a^*)E(z)),
\]
which vanishes because $\tau$ is a trace on $C^*(S_c)$. We have shown that for every $[s] \in S/_\sim, [s] \neq S_c$ the contribution $\chi_{\tau,s}(q_{I}\varphi(x)q_{I})$ vanishes for all large enough $I$, and we conclude thus from \eqref{eq:parametrization is surj} that  $(\psi_{\beta,\tau})_{\text{Irr}(N(S))}(\varphi(x)) = \tau(x)$ for all $x \in C^*(S_c)$.
\end{proof}

%\begin{remark}\label{rem:trace on level n}
%Every KMS-state constructed from a normalised trace $\tau$ on $C^*(S_c)$ via Proposition~\ref{prop:construction of KMS-states} for $\beta >1$ is the weak* limit obtained from averaging normalised traces $\tau_n := n^{-1}\sum_{[s] \in N^{-1}(n)/_\sim} \chi_{\tau,s}$ on $C^*(S_c)$ via  $\zeta_I(\beta)^{-1}\sum_{n \in \langle I\rangle^+} n^{1-\beta} \tau_n$ (with $I \nearrow \text{Irr}(N(S))$, and $\beta \searrow 1$ for inverse temperature $1$), see \eqref{eq:trace on level n}. This is mirrored by the decomposition of the core Fock module as $\FM = \bigoplus_{n \in N(S)} \FM_n$ with $\FM_n := \overline{\bigoplus_{[s] \in N^{-1}(n)/_\sim} \FM_{0,[s]}/\FN}^{\|\cdot\|}$ as $\tau_n$ is constructed via $\FM_n$.\footnote{Nadia, do you remember the reference we wanted to relate this to? I remember we talked about this a while ago at Ulleval, and I think it was related to product systems over $\N$, so that $N(S)$ would be replaced by the simpler datum $\N$ with a single $X^{\otimes n}$ as summand. Do you recall that?}
%\end{remark}

%\subsection{Thermally stable fixed points}\label{subsec:thermally stable FP}

%\begin{thm}\label{thm:kms states via functional equation}
%For every $\beta_0 \geq 1$, the simplex $T_{>\beta_0}(C^*(S_c))$ embeds into the simplex of KMS$_{\beta_0}$-states via %$\tau \mapsto \psi_{\beta_0,\tau}$. This map is surjective for $\beta_0=1$ with inverse given by the restriction $\phi %\mapsto \phi\circ\varphi$.
%\end{thm}
\begin{proof}[Proof of Theorem~\ref{thm:restricting to tau}]
Let $\beta_0 \geq 1$ and suppose $\tau$ is a normalised trace on $C^*(S_c)$ satisfying \eqref{eq:KMS funct eq}. By weak*- compactness of the state space, a subnet of $(\psi_{\beta,\tau})_{\beta >\beta_0}$, where the index set is ordered by the reverse order, converges to a state $\phi_{\beta_0,\tau}$. It is evident that $\phi_{\beta_0,\tau}$ is a KMS$_{\beta_0}$-state with $\phi_{\beta_0,\tau}\circ \varphi = \tau$. By Proposition~\ref{prop:alg char of KMS-states}, the latter determines the KMS$_{\beta_0}$-state completely, so every convergent subnet converges to this state and $\phi_{\beta_0,\tau}=\psi_{\beta_0,\tau}$. Moreover, $\phi_{\beta_0,\tau}\circ \varphi = \tau$ shows that we get an embedding of the simplex of normalised traces satisfying \eqref{eq:KMS funct eq} into the simplex of KMS$_{\beta_0}$-states.

Now suppose $\beta_0=1$ and let $\phi$ be a KMS$_1$-state. In order to show that the embedding obtained before is a surjection, we will show that $\psi_{\beta,\phi\circ \varphi}\circ \varphi = \phi\circ \varphi$ for all $\beta >1$. To do so, we claim first that for $x \in C^*(S_c)$ we have $v_s^*\varphi(x)v_s^{\phantom{*}} \in C^*(S_c)$ for all $s \in S$. To see this, it suffices to consider  $x=w_a^{\phantom{*}}w_b^*$ with $a,b \in S_c$ such that $0 \neq v_s^*\varphi(x)v_s^{\phantom{*}}$. The latter amounts to $\alpha_a^{-1}([s])=\alpha_b^{-1}([s])$, so that there exist $c,d \in S_c$ with $v_s^*\varphi(x)v_s^{\phantom{*}} = v_c^{\phantom{*}}v_d^*$. Hence
\[\chi_{\phi\circ\varphi,s}(\varphi(x)) = \phi\circ\varphi\circ E(v_s^*\varphi(x)v_s^{\phantom{*}}) =  \phi(v_s^*\varphi(x)v_s^{\phantom{*}}).\]
Invoking Lemma~\ref{lem:construction of KMS-states}, we get $\psi_{\phi\circ\varphi,n}\circ\varphi = n^{-1}\sum_{[s] \in N^{-1}(n)/_\sim} \chi_{\phi\circ\varphi,s}\circ\varphi \stackrel{\eqref{eq:rec formula beta=1}}{=} \phi\circ\varphi$, and then
\[\begin{array}{c}
\psi_{\beta,\phi\circ\varphi}\circ \varphi = \lim\limits_{I\nearrow \text{Irr}(N(S))} \zeta_I(\beta)^{-1}\sum\limits_{n \in \langle I\rangle^+} n^{1-\beta} \psi_{\phi\circ\varphi,n}\circ\varphi = \phi \circ\varphi
\end{array}\]
for all $\beta>1$. Therefore, the normalised trace $\phi\circ\varphi$ satisfies \eqref{eq:KMS funct eq} for $\beta_0=1$, and $\phi_{1,\phi\circ\varphi}\circ \varphi = \phi\circ\varphi$ entails $\phi_{1,\phi\circ\varphi} = \phi$ because of Proposition~\ref{prop:alg char of KMS-states}.
\end{proof}

\begin{corollary}\label{cor:kms states via functional equation} The following hold:
\begin{enumerate}[(i)]
\item For $1\leq \beta_1 \leq \beta_2$, the simplex $T_{>\beta_1}(C^*(S_c))$ is a face in $T_{>\beta_2}(C^*(S_c))$.
\item For every $\beta>1$, the simplex of KMS$_1$-states embeds into the simplex of KMS$_\beta$-states.
\item If $T_{>\beta_1}(C^*(S_c))$ is not a singleton, then $(C^*(S),\sigma)$ does not have a unique KMS$_{\beta_2}$-state for all $\beta_2\geq \beta_1$.
\end{enumerate}
\end{corollary}
\begin{proof}
The first claim is just an observation, which in combination with Theorem~\ref{thm:restricting to tau} proves both (ii) and (iii).
\end{proof}

\begin{question}\label{que:kms states via functional equation}
Under which conditions can one extend the characterisation of the KMS$_{\beta_0}$-states for $\beta_0=1$ from Theorem~\ref{thm:restricting to tau}  to other values of $\beta_0$, say $\beta_0 \in (1,\beta_c]$?
\end{question}

\section{Uniqueness for KMS-states in the critical interval}\label{sec:uniqueness}
Let $S$ be a right LCM monoid with generalised scale $N\colon S\to \N^\times$. It follows from Proposition~\ref{prop:construction of KMS1-states} and Proposition~\ref{prop:construction of KMS-states} that every $\beta$ in the critical interval $[1,\beta_c]$ has at least one KMS$_\beta$-state on $(C^*(S),\sigma)$. In this section, we will address sufficient conditions for uniqueness of these states.

Our first criterion is the notion of core regularity introduced in Definition~\ref{def:neg q-a}. This is pertinent to uniqueness of KMS$_\beta$-state when the critical interval degenerates to the value $\beta=1$. In order to justify our definition, first observe that if $\phi$ is a $\kms_1$-state, $a,b\in S_c$ and $n\in N(S)$, then by equation \eqref{eq:rec formula beta=1} and  the KMS$_1$-condition we obtain
\begin{equation}\label{eq:sum-matters}
\begin{array}{c}
  \phi(v_a^{\phantom{*}}v_b^*) = \phi(v_b^*v_a^{\phantom{*}}) = \sum\limits_{[s] \in N^{-1}(n)/_\sim}\hspace*{-4mm} n^{-1}\phi(v_{bs}^{*}v_{as}^{\phantom{*}})
   = \sum\limits_{[s] \in N^{-1}(n)/_\sim}\hspace*{-4mm} n^{-1} \delta_{[as], [bs]} \ \phi(v_{bs}^*v_{as}^{\phantom{*}}).
\end{array}
\end{equation}
   The notion of core regularity says that
 for every choice of distinct elements $a,b \in S_c$, there may only be a small proportion of elements $[s] \in S/_\sim$ such that $[as] = [bs]$ but $asc \neq bsc$ for all $c \in S_c$.

Fix $a,b\in S_c$ and suppose $s\in S$ such that $[as]=[bs]$. We claim that if for some $s'\in [s]$ there is $c\in S_c$ such that $as'c=bs'c$, then there exists $d\in S_c$ such that $asd=bsd$. This will say that the property of absorbing the difference between $a$ and $b$ is invariant under $\sim$, and hence a property of $[s]$. To prove the claim, let $sS\cap s'S= seS, se=s'f$. Since $S_c$ is itself right LCM, we can further take $fS_c\cap cS_c=fgS, fg=ch$ with $g,h\in S_c$. Then we get $aseg=as'fg=as'ch=bs'ch=bs'fg=bseg$, so that $d:= eg$ yields the claim. Diagramatically, we can illustrate this as follows, with the bottom full arrows arising from the first equivalence $as\sim bs$:
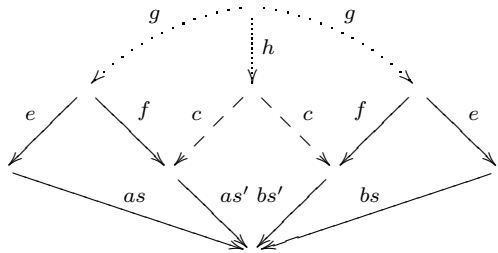
\[ \xymatrix{
&&& \ar@/_/@{.>}[dll]_g \ar@/^/@{.>}[drr]^g \ar@{.>}[d]^h \\
%&&&\\
&\ar[dl]_e \ar[dr]^f&&\ar@{-->}[dl]_c \ar@{-->}[dr]^c&&\ar[dl]_f \ar[dr]^e\\
\ar[drrr]^{as} &&\ar[dr]^{as'}&&\ar[dl]_{bs'}&&\ar[dlll]_{bs}\\
&&&
}\]

\begin{remark}\label{rem:neg q-a for core being a group}
Let $a,b,c \in S_c$ such that $a=bc$. Then left cancellation implies $F_n^{a,b} = F_n^{c,1}$ and  $A_n^{a,b} = A_n^{c,1}$. If instead we have  $a=cb$, then $[s]\in F_n^{a,b}$ is equivalent to
$[bs] \in F_n^{c,1}$. Hence $F_n^{a,b}=\alpha_b^{-1}(F_n^{c,1})$. Likewise, $A_n^{a,b} = \alpha_b^{-1}(A_n^{c,1})$. It follows that if all pairs in $S_c$ are comparable with respect to one of the partial orders given by $a \geq_r b \Leftrightarrow a \in bS_c$ and $a\geq_\ell b\Leftrightarrow a \in S_cb$, e.g.~if $S_c$ is a group or $\N$, then it suffices to consider $\lvert F_n^{a,1}\setminus A_n^{a,1}\rvert$ for arbitrary $a \in S_c$ in order to determine $\lvert F_n^{a,b}\setminus A_n^{a,b}\rvert$ for all $a,b \in S_c$. In this case, we shall simplify the notation to $F_n^a:= F_n^{a,1}$ and $A_n^a:= A_n^{a,1}$.
\end{remark}

\begin{remark}\label{rem:kappa monotone increasing}
The sequence $(\lvert A_n^{a,b}\rvert/n)_{n \in N(S)} \subset [0,1]$ is monotone increasing for all $a,b \in S_c$ in the sense that if $[s] \in A_m^{a,b}$ and $t \in N^{-1}(n)$, then $[st]\in A_{mn}^{a,b}$. To see this, pick $e \in S_c$ satisfying $ase=bse$, then $tS\cap eS=tfS, tf=et'$ for some $f \in S_c, t' \in S$. This implies $astf = aset' = bset' = bstf$, as needed. %Thus the limit $\kappa_{a,b} := \lim_{n \in N(S)} \kappa_{a,b,n} \in [0,1]$ exists for all $a,b \in S_c$.
\end{remark}

Self-similar actions $(G,X)$ illustrate the meaning of the sets $A_n^{a,b}\subset F_n^{a,b}$ very nicely:

\begin{proposition}\label{prop:neg q-a explained for SSA}
Suppose $S=X^*\bowtie G$ for a self-similar action $(G,X)$. For $g \in G$, the equivalence class $[(w,1_G)] \in S/_\sim$ for a word $w \in X^*$ with $n:= \lvert X\rvert^{\ell(w)}$ belongs to $F_n^{(\varnothing,g)}$ if and only if $g(w) = w$. It belongs to $A_n^{(\varnothing,g)}$ if and only if $g(w)=w$ and $g|_w=1_G$.
\end{proposition}
\begin{proof}
We have $S_c = \{\varnothing\}\times G$, and $[(w,h)] \in S/_\sim$ is determined by $w$. Therefore $\alpha_{(\varnothing,g)}([(w,1_G)]) = [(g(w),g|_w)]$ shows the first claim. For the second part, one direction is obvious. So suppose there is $h\in G$ such that $(w,g|_wh)=(\varnothing,g)(w,h) = (w,h)$. By left cancellation we get $g|_wh=h$, which amounts to $g|_w=1_G$ since $h$ is invertible.
\end{proof}

\begin{proposition}\label{prop:uniqueness at 1 if neg q-a}
There is a $\kms_1$-state $\psi_1$ for $(C^*(S),\sigma)$  determined by
\[\psi_1(v_a^{\phantom{*}}v_b^*) = \lim\limits_{n \in N(S)} \frac{\lvert A_n^{a,b}\rvert}{n} \quad \text{for } a,b \in S_c.\]
If $S$ is core regular, then $\psi_1$ is the unique $\kms_1$-state.
\end{proposition}
\begin{proof}
According to Proposition~\ref{prop:trace on C*(S_c)}, $\tau_0$ given by $\tau_0(v_a^{\phantom{*}}v_b^*) =\delta_{a,b}$ defines a normalised trace on $C^*(S_c)$. By Proposition~\ref{prop:construction of KMS1-states}, this yields a $\kms_1$-state $\psi_1:= \psi_{1,\tau_0}$. For $a,b \in S_c$ and $n \in N(S)$, we compute
\[\begin{array}{c}
\psi_{\tau_0,n}(v_a^{\phantom{*}}v_b^*) = \psi_{\tau_0,n}(v_b^*v_a^{\phantom{*}}) = n^{-1}\sum\limits_{[r] \in N^{-1}(n)/_\sim} \tau_0\circ E(v_{br}^*v_{ar}^{\phantom{*}}).
\end{array}\]
Since $\tau_0\circ E(v_{br}^*v_{ar}^{\phantom{*}})$ vanishes unless $arc=brc$ for some $c \in S_c$, we arrive at
\begin{equation} \label{eq:unique KMS1-state}
\psi_{\tau_0,n}(v_a^{\phantom{*}}v_b^*) = n^{-1} \lvert A_n^{a,b}\rvert,
\end{equation}
which immediately implies the claim for $\psi_1$. Now suppose $\phi$ is any $\kms_1$-state on $C^*(S)$. Let $a,b\in S_c$ and $n\in N(S)$. According to \eqref{eq:sum-matters},
\[\begin{array}{lclclcl}
\phi(v_a^{\phantom{*}}v_b^*) &=& \sum\limits_{[s] \in N^{-1}(n)/_\sim} n^{-1} \delta_{[as], [bs]} \phi(v_{bs}^*v_{as}^{\phantom{*}})
&=& \sum\limits_{[s] \in F_n^{a,b}} n^{-1}\phi(v_{bs}^*v_{as}^{\phantom{*}}) \vspace*{2mm}\\
&=&  \sum\limits_{[s] \in A_n^{a,b}}n^{-1}{\phi(v_{bs}^*v_{as}^{\phantom{*}})} + \sum\limits_{[s] \in F_n^{a,b}\setminus A_n^{a,b}} n^{-1}{\phi(v_{bs}^*v_{as}^{\phantom{*}})}.
\end{array}\]
By Proposition~\ref{prop:alg char of KMS-states}, $\phi(v_{bs}^*v_{as}^{\phantom{*}})$ has form $\phi(v_{g}^{\phantom{*}}v_{f}^*)$ for some $f,g\in S_c$. Each summand corresponding to $[s]\in A_n^{a,b}$ will satisfy $g=f$, and since $w_g$ is unitary in $C^*(S_c)$ and $\phi\circ \varphi$ is a trace on $C^*(S_c)$ we have $\phi(v_{g}^{\phantom{*}}v_{f}^*)=1$. For the summands where $[s]\in F_n^{a,b}\setminus A_n^{a,b}$, we  estimate that $|\phi(v_{g}^{\phantom{*}}v_{f}^*)|\leq 1$. Thus the assumption on $S$ implies  $\phi = \psi_1$, as needed.
\end{proof}

Proposition~\ref{prop:uniqueness at 1 if neg q-a} improves \cite{ABLS}*{Proposition~9.2} significantly. If $S$ satisfies the hypotheses of \cite{ABLS}*{Proposition~9.2}, namely faithfulness of $\alpha\colon S_c \curvearrowright S/_\sim$ and finite propagation, then \cite{ABLS}*{Lemma~9.5} establishes that $S$ is core regular. In Section~\ref{sec:unique despite unfaithful alpha} we present an example that has a unique KMS$_1$-state (with $\beta_c=1$) and for which $\alpha$ is not faithful.

The value $\beta_c$, when finite, can be interpreted as a critical inverse temperature above which KMS$_\beta$-states are of Gibbs type in the sense of \cite{ALN}. The value $\beta_1=1$, on the other hand, has the relevance of an inverse temperature below which no KMS$_\beta$-states can exist. The question of deciding in general when these values $\beta_1$ and $\beta_c$ coincide is a very interesting one and far from fully answered, see though \cite{ALN}.

There are  classes of examples where $\beta_1<\beta_c$, and for this situation it is interesting to decide if there is uniqueness of  KMS$_\beta$-states for $\beta\in (\beta_1,\beta_c]$. Our second criterion, that of summably core regular monoid  from Definition~\ref{def:sum q-a}, is tailored to the case of $\beta \in (1,\beta_c]$, which occurs for instance for the affine semigroup over the natural numbers from \cite{LR2}. Up to now, it was only almost freeness of $\alpha$ that was used to detect uniqueness in this situation, compare \cite{ABLS}*{Proposition~9.1}.

 \begin{lemma}\label{lem:kappa monotone increasing for s q-a}
Let $\beta\in (1,\beta_c]$ and $a,b\in S_c$. The sequence $\{\zeta_I(\beta)^{-1} \sum_{n \in \langle I\rangle^+} n^{-\beta} \lvert A_n^{a,b} \rvert\}$ ranging over $I\ssubset \text{Irr}(N(S))$ is monotone increasing, and hence converges to a limit
\[\begin{array}{c}
\kappa_{\beta,a,b}:=  \lim\limits_{I \ssubset \text{Irr}(N(S))} \zeta_I(\beta)^{-1} \sum\limits_{n \in \langle I\rangle^+} n^{-\beta} \lvert A_n^{a,b}\rvert \in [0,1].
\end{array}\]
 \end{lemma}
\begin{proof}
Suppose $I,J \ssubset \text{Irr}(N(S))$ are disjoint, then $\zeta_{I\cup J}(\beta)=\zeta_I(\beta)\zeta_J(\beta)$ by the definition of the restricted $\zeta$-function in \eqref{eq:restricted-zetaI}. Since $\lvert A_{mn}^{a,b}\rvert \geq \lvert A_{m}^{a,b}\rvert n$ for all $m,n \in N(S)$ due to  Remark~\ref{rem:kappa monotone increasing}, we estimate that
\[\begin{array}{lcl}
\zeta_{I\cup J}(\beta)^{-1}\sum\limits_{k \in \langle I\cup J\rangle^+} k^{-\beta} \lvert A_k^{a,b}\rvert
&=& \zeta_I(\beta)^{-1} \sum\limits_{m \in \langle I\rangle^+} m^{-\beta} \zeta_J(\beta)^{-1}\sum\limits_{n \in \langle J\rangle^+} n^{-\beta} \lvert A_{mn}^{a,b}\rvert \vspace*{2mm}\\
&\geq& \zeta_I(\beta)^{-1} \sum\limits_{m \in \langle I\rangle^+} m^{-\beta} \lvert A_m^{a,b}\rvert \zeta_J(\beta)^{-1}\sum\limits_{n \in \langle J\rangle^+} n^{-\beta+1} \vspace*{2mm}\\
&=& \zeta_I(\beta)^{-1} \sum\limits_{m \in \langle I\rangle^+} m^{-\beta} \lvert A_m^{a,b}\rvert.
\end{array}\]
Since Thus the sequence is monotone increasing as claimed, and since all its terms lie in $[0,1]$ the lemma follows.
\end{proof}

\begin{proposition}\label{prop:uniqueness in the crit interval by s q-a}
There is a $\kms_\beta$-state $\psi_\beta$ for $(C^*(S),\sigma)$ for each  $\beta \in (1,\beta_c]$ determined by
\[\begin{array}{c}
\psi_\beta(v_a^{\phantom{*}}v_b^*) = \lim\limits_{I \ssubset \text{Irr}(N(S))} \zeta_I(\beta)^{-1} \sum\limits_{n \in \langle I\rangle^+} n^{-\beta} \lvert A_n^{a,b}\rvert \quad \text{for } a,b \in S_c.
\end{array}\]
If $S$ is $\beta$-summably core regular, then $\psi_\beta$ is the unique the KMS$_\beta$-state.
\end{proposition}
\begin{proof}
Let $\beta \in (1,\beta_c]$. As in the proof of Proposition~\ref{prop:uniqueness at 1 if neg q-a}, the normalised trace $\tau_0$ given by $\tau_0(v_a^{\phantom{*}}v_b^*) =\delta_{a,b}$ gives rise to $\kms_\beta$-state $\psi_\beta:= \psi_{\beta,\tau_0}$ using Proposition~\ref{prop:construction of KMS-states}. By construction, $\psi_\beta$ is the weak* limit of the $\psi_{\beta,\tau,I} = \zeta_I(\beta)^{-1} \sum_{n \in \langle I\rangle^+} n^{1-\beta}\psi_{\tau,n}$ with $I\ssubset \text{Irr}(N(S))$. Therefore, \eqref{eq:unique KMS1-state} shows that $\psi_\beta$ has the claimed form.

Now let $\phi$ be any KMS$_\beta$-state and $a,b\in S_c$ be fixed but arbitrary. The KMS$_\beta$-condition gives $\phi(v_a^{\phantom{*}}v_b^*)=\phi(v_b^*v_a^{\phantom{*}})$, and we note that for all $s \in S$, the expression $v_{bs}^*v_{as}^{\phantom{*}}$ vanishes unless $[s] \in F_{N_s}^{a,b}$. Hence, by Lemma~\ref{lem:rec formula}~(iv), we have
\[\begin{array}{lcl}
\phi(v_a^{\phantom{*}}v_b^*) &=& \lim\limits_{I \ssubset \text{Irr}(N(S))} \zeta_I(\beta)^{-1} \sum\limits_{[s] \in S_I/_\sim} N_s^{-\beta} \phi_I(v_{bs}^*v_{as}^{\phantom{*}})\vspace*{2mm}\\
&=&\lim\limits_{I \ssubset \text{Irr}(N(S))} \zeta_I(\beta)^{-1} \sum\limits_{n \in \langle I\rangle^+}\sum\limits_{[s] \in F_n^{a,b}} n^{-\beta} \phi_I(v_{bs}^*v_{as}^{\phantom{*}}).
\end{array}\]
Now we split the inner sum over $[s]\in F_n^{a,b}\setminus A_n^{a,b}$ and $[s]\in A_n^{a,b}$ and use that $\phi_I(v_{bs}^*v_{as}^{\phantom{*}}) = 1$ whenever $[s] \in A_n^{a,b}$ as $\phi_I\circ \varphi$ is tracial, which can be seen similar to the proof of Proposition~\ref{prop:uniqueness at 1 if neg q-a}. Therefore invoking summable quasi-absorption at $\beta$ we obtain
\[\begin{array}{lcl}
|\phi(v_a^{\phantom{*}}v_b^*)-\kappa_{\beta,a,b}|
&=& \lim\limits_{I \ssubset \text{Irr}(N(S))} \zeta_I(\beta)^{-1} \sum\limits_{n \in \langle I\rangle^+}\sum\limits_{[s] \in F_n^{a,b}\setminus A_n^{a,b}} n^{-\beta} \underbrace{|\phi_I(v_{bs}^*v_{as}^{\phantom{*}})|}_{\leq 1} \vspace*{2mm}\\
&\leq&  \lim\limits_{I \ssubset \text{Irr}(N(S))} \zeta_I(\beta)^{-1} \sum\limits_{n \in \langle I\rangle^+}n^{-\beta} \lvert F_n^{a,b}\setminus A_n^{a,b}\rvert  = 0.
\end{array}\]
Since Proposition~\ref{prop:alg char of KMS-states} says that $\phi$ is determined by its values on monomials $v_a^{\phantom{*}}v_b^*$, we conclude  that there is a unique such state $\phi=\psi_\beta$.
\end{proof}

\begin{proposition}\label{prop:neg q-a vs. other conditions}
Suppose $S$ is a right LCM monoid with generalised scale. If $\alpha$ is almost free, then $S$ is core regular and summably core regular.
\end{proposition}
\begin{proof}
The set $S/_\sim$ is infinite because the generalised scale is nontrivial by definition. If $\alpha$ is almost free, then, for given $a,b \in S_c, a \neq b$, the set $F_n^{a,b}$ is empty for almost all $n \in N(S)$. This implies that  $S$ is core regular. Next, suppose we have $\beta\in (1,\beta_c]$ and $a,b\in S_c$. Since $0 < n^{-\beta} \leq 1$, the series $\sum_{n \in \langle I\rangle^+} n^{-\beta} \lvert F_n^{a,b} \setminus A_n^{a,b} \rvert$ is bounded by the total number of fixed points for $\alpha_a^{-1}\alpha_b$, that is, $\sum_{n \in N(S)} \lvert F_n^{a,b} \rvert <\infty$. But $\zeta_I(\beta) \nearrow \infty$ as $I\nearrow \text{Irr}(N(S))$ because $\beta \in (1,\beta_c]$, so $S$ is $\beta$-summably $(a,b)$-regular. Since $\beta$, $a$ and $b$ were arbitrary, we get the claim.
\end{proof}

In Section~\ref{subsec:q-a but not almost free} we will discuss a class of examples with $\beta_c=2$ that satisfy both types of core regularity, but for which $\alpha$ is not almost free.

\begin{proposition}\label{prop:r canc: neg q-a forces alpha faithful}
Suppose $S$ is right cancellative. Then $A_n^{a,b} = \emptyset$ holds for all $a\neq b$ and $n\in N(S)$. If $S$ is core regular, then $\alpha$ is faithful.
\end{proposition}
\begin{proof}
The first claim follows from the definition. For the second part, suppose there are $a,b \in S_c, a\neq b$ with $\alpha_a=\alpha_b$. Then $F_n^{a,b}=N^{-1}(n)/_\sim$ for all $n \in N(S)$, thus $\vert F_n^{a,b}\vert=n$ and $S$ fails to be $(a,b)$-regular. By Definition~\ref{def:neg q-a}, $S$ is not core regular.
\end{proof}

Finally, we establish a sufficient condition for $S$ to be summably core regular. We start with a lemma, whose third part is a generalisation of \cite{ABLS}*{Equation (9.1)}.

\begin{lemma}\label{lem:factorization formula for fixed non-absorbing elements}
Let $a,b \in S_c$ and $m,n \in N(S)$. Then:%check references to items below
\begin{enumerate}[(i)]
\item $F_{mn}^{a,b} = \{ [st] \mid [s] \in F_m^{a,b}, [t] \in \alpha_c(F_n^{c,d}) \text{ for } c,d \in S_c, asS\cap bsS = ascS, asc=bsd\}$. \vspace*{-3mm}
\item $[s] \in A_m^{a,b}$ implies $[st] \in A_{mn}^{a,b}$ for all $t \in N^{-1}(n)$.\vspace*{1mm}
\item $F_{mn}^{a,b}\setminus A_{mn}^{a,b} = \{ [st] \mid [s] \in F_m^{a,b}\setminus A_m^{a,b} \text{ and } [t] \in \alpha_c(F_n^{c,d}\setminus A_n^{c,d}) \}$.
\end{enumerate}
\end{lemma}
\begin{proof}
For (i), we first note that for all $m,n \in N(S)$, every element $[r] \in N^{-1}(mn)/_\sim$ is of the form $[st]$ for some $s \in N^{-1}(m), t \in N^{-1}(n)$. This follows from \cite{ABLS}*{Definition~3.1(A3)(b) and Proposition~3.6(iv)}. Let $s,t\in S$ be chosen accordingly for a fixed, but arbitrary $[r]$ as above with $[r]=[st] \in F_{mn}^{a,b}$, that is, there are $e,f \in S_c$ with $astS\cap bstS = asteS, aste= bstf$. Thus $as \not\perp bs$, and since  $N_{as}=N_s=N_{bs}$ we obtain by \cite{ABLS}*{Proposition~3.6(ii)} that $as\sim bs$. That is, $[s] \in F_m^{a,b}$. Thus there are $c,d \in S_c$ with $asS\cap bsS = ascS, asc=bsd$, for which we can find $t',t'' \in S$ with
\[asct' = aste = bstf = bsdt''\]
because $aste=bstf \in asS \cap bsS$. By $asc=bsd$ and left cancellation, we get $t''=t'$. Using left cancellation for $as$ and $bs$, we further deduce that $\alpha_c([t']) = [t]= \alpha_d([t'])$, that is, $[t] \in \alpha_c(F_n^{c,d})$. This shows ``$\subset$''.

To prove the reverse inclusion, let $[s] \in F_m^{a,b}$ and $[t] \in \alpha_c(F_n^{c,d})$ for $c,d \in S_c$ as prescribed in (i). Thus for $t' \in S$ with $[t'] = \alpha_c^{-1}([t])$, we get $ast \sim asct' = bsdt' \sim bst$, that is, $[st] \in F_{mn}^{a,b}$.

For (ii), suppose there exists $e \in S_c$ such that $ase=bse$. Given $t \in N^{-1}(n)$, there are $t' \in S,f \in S_c$ with $eS\cap tS=et'S, et'=tf$. This yields $astf=aset'=bset'=bstf$, that is, $[st] \in A_{mn}^{a,b}$.

For part (iii), we claim that for $[st] \in F_{mn}^{a,b}$, $[st] \in A_{mn}^{a,b}$ is equivalent to $[t] \in \alpha_c(A_n^{c,d})$: If there is $e \in S_c$ such that $aste=bste$, then $aste = asct'=bsdt'$ for some $t' \in S$ with $ct'\sim t\sim dt'$ (by left cancellation), that is, $[t] \in \alpha_c(F_n^{c,d})$. Conversely, if there is $t' \in S$ with $ct' \sim t$ such that $ct'=dt'$, then $ct'=te$ for some $e \in S_c$ and $aste=asct'=bsdt'=bste$. Together with (i) and (ii), this yields (iii).
\end{proof}

\begin{remark}\label{rem:factorization formula for fixed non-absorbing elements}
The set $\alpha_c(F_n^{c,d})$ appearing in Lemma~\ref{lem:factorization formula for fixed non-absorbing elements}~(i) may depend on the choice of the representative $s \in [s]$, but the resulting sets are related via $\alpha$: Let $a,b \in S_c$ and $m,n \in N(S)$, $[s] \in F_m^{a,b}$ and $s'\sim s$, say $s'S\cap sS = seS, se=s'f$. If we denote by $c,d,c',d' \in S_c$ elements that satisfy $asS\cap bsS=ascS, asc=bsd$ and $as'S\cap bs'S=as'c'S, asc'=bsd'$, then
\[\alpha_{c'}(F_n^{c',d'}) = \alpha_f\alpha_e^{-1}\alpha_c(F_n^{c,d}).\]
Indeed, for every $[t] \in F_n^{c,d}$ there is $t' \in S$ with $sctS\cap s'S = s(ctS\cap eS) = set'S$ and $et' \sim ct \sim dt$, which leads to
\[as'ft'= aset' \sim asct = bsdt \sim bsct \sim bset'=bs'ft'.\]
This yields $\alpha_f\alpha_e^{-1}\alpha_c(F_n^{c,d}) \subset \alpha_{c'}(F_n^{c',d'})$. Using the symmetry of the argument in $s$ and $s'$, and finiteness of the involved sets, we conclude that $\alpha_f\alpha_e^{-1}\alpha_c(F_n^{c,d}) = \alpha_{c'}(F_n^{c',d'})$.
\end{remark}

\begin{corollary}\label{cor:suff for q-a}
Let $S$ be a right LCM monoid $S$ with generalised scale $N$ and $a,b \in S_c$. If there exist a finite subset $E \ssubset \text{Irr}(N(S))$ and a constant $C>0$ such that $\lvert F_n^{a,b}\setminus A_n^{a,b}\rvert \leq C$ for all $n \in \langle \text{Irr}(N(S))\setminus E\rangle^+$, then $S$ is $\beta$-summably $(a,b)$-regular for every $\beta \in (\beta_c-1,\beta_c]$.
\end{corollary}
\begin{proof}
Fix $\beta \in (\beta_c-1,\beta_c]$ and $I\ssubset \text{Irr}(N(S))$. Decomposing each $n\in \langle I\rangle^+$ as $n=n_1n_2$ with $n_1 \in \langle I\setminus E\rangle^+$ and $n_2 \in \langle I\cap E\rangle^+$, we use Lemma~\ref{lem:factorization formula for fixed non-absorbing elements}(iii)  to estimate that
\[\begin{array}{lcl}
\sum\limits_{n \in \langle I\rangle^+} n^{-\beta} \lvert F_n^{a,b}\setminus A_n^{a,b} \rvert
&=& \sum\limits_{n \in \langle I\rangle^+} n_1^{-\beta} \sum\limits_{[s] \in F_{n_1}^{a,b}\setminus A_{n_1}^{a,b}} n_2^{-\beta}\lvert F_{n_2}^{c(s),d(s)}\setminus \alpha_{c(s)}(A_{n_2}^{c(s),d(s)}) \rvert \vspace*{2mm}\\
&\leq& C\zeta_{I\setminus E}(\beta+1)\zeta_{I\cap E}(\beta) \leq C\zeta_{\text{Irr}(N(S))\setminus E}(\beta+1)\zeta_E(\beta)<\infty.
\end{array}\]
On the other hand, we have $\zeta_I(\beta) \nearrow \infty$ as $I\nearrow \text{Irr}(N(S))$, and so   \eqref{eq:ab-summable-qa} is satisfied, as required for $\beta$-summably $(a,b)$-regular.
\end{proof}

\section{Positivity breaking in Baumslag--Solitar monoids}\label{subsec:BS3}
For nonzero integers $c$ and $d$, the Baumslag--Solitar group is the universal group $BS(c,d):= \langle \ia,\ib \mid \ia\ib^c=\ib^d\ia\rangle$ introduced in the 1960's in the influential note \cite{BaSo}.
 Here we are interested in the submonoid of $BS(c,d)$ generated by $\ia$ and $\ib$, the \emph{Baumslag--Solitar monoid} $BS(c,d)^+$. Since these are finitely generated one-relator monoids, Adjan's criterion for embeddability of universal semigroups defined by generators and relations into the corresponding group applies, see \cite[{Section II, Theorem 3}]{AdjTrans}. Thus $BS(c,d)^+$ is the universal monoid $\langle \ia,\ib \mid  \ia\ib^c=\ib^d\ia\rangle^+$ for $cd>0$, and $\langle \ia,\ib \mid  \ib^{|d|}\ia\ib^{|c|}=\ia\rangle^+$ for $cd<0$.

We record some facts about $BS(c,d)^+$ from \cite{Spi1}. The \emph{height} is the  homomorphism $\theta\colon BS(c,d) \to \Z$ determined by $\ia \mapsto 1, \ib \mapsto 0$. By \cite[{Proposition~2.3 (L)}]{Spi1}, each $s \in BS(c,d)^+$ admits a unique normal form
\begin{equation}\label{eq:BS normal form}
s= \ib^{i_1}\ia\ib^{i_2}\ia\ldots\ia\ib^{i_{\theta(s)}}\ia\ib^m, \quad \text{where } 0\leq i_1,\ldots,i_{\theta(s)} < |d|
\end{equation}
and $m \in \N$ in case $cd>0$, while in the case $cd<0$ and $\theta(s)>0$ we have $m \in \Z$. Thus in particular, if $cd>0$, then an element $s$ of $ BS(c,d)^+$ with trivial height must have normal form $b^m$ with $m\in \N$.
 This normal form allows for efficient computation of right common multiples, see \cite{Spi1}*{Proposition~2.10}: two elements $s,t \in BS(c,d)^+$ with normal forms $s= \ib^{i_1}\ia\ib^{i_2}\ia\ldots\ia\ib^{i_{\theta(s)}}\ia\ib^m$ and $t= \ib^{j_1}\ia\ib^{j_2}\ia\ldots\ia\ib^{j_{\theta(t)}}\ia\ib^n$ satisfy $s\Cap t$ if and only if $i_k=j_k$ for all $1\leq k \leq \min(\theta(s),\theta(t))$. If this is the case, then there is a unique right LCM for $s$ and $t$. In particular, it follows that $BS(c,d)^+$ is a right LCM semigroup for all $c,d \in \Z^\times$.

Spielberg proceeds by showing that the pair $(BS(c,d),BS(c,d)^+)$ is quasi lattice-ordered if and only if $cd>0$, see \cite{Spi1}*{Theorem~2.11}. The quasi lattice-order was the essential prerequisite for Clark--an Huef--Raeburn \cite{CaHR} to study the KMS-states for $C^*(BS(c,d)^+)$ with $cd>0$ via the classical Nica-Toeplitz algebra construction.

Our aim here is to complete the classification of KMS-states for the semigroup $C^*$-algebras of Baumslag--Solitar monoids, by not only incorporating the case that $cd<0$, but also fully describing the simplex of KMS$_\beta$-states at the critical value $\beta=1$. We summarise first some features we shall need in the sequel.

\begin{proposition}\label{prop:BS features}
Let $c,d \in \IZ^\times$.
\begin{enumerate}[(i)]
\item The core $(BS(c,d)^+)_c$ is canonically identified with the monoid $\langle \ib\rangle^+ \cong \N$ if $|d|>1$, and coincides with $BS(c,d)^+$ if $d=\pm 1$.
\item There exists a generalised scale $N$ on $BS(c,d)^+$ if and only if $|d|>1$, in which case $N$ is given by $s \mapsto |d|^{\theta(s)}$ for $s \in BS(c,d)^+$.
\item For $|d|>1$, two elements $s,t \in BS(c,d)^+$ satisfy $s\sim t$ if and only if $\theta(s)=\theta(t)=:k\geq 1$ and, if $k>1$, the respective normal forms $s= \ib^{i_1}\ia\ib^{i_2}\ia\ldots\ia\ib^{i_k}\ia\ib^m$ and $t= \ib^{j_1}\ia\ib^{j_2}\ia\ldots\ia\ib^{j_k}\ia\ib^n$ for $m,n\in \Z$ satisfy $i_\ell=j_\ell$ for all $1\leq \ell\leq k$.
\end{enumerate}
\end{proposition}
\begin{proof}
Part (i) follows from \eqref{eq:BS normal form}, as does (ii) when \eqref{eq:BS normal form} is combined with \cite{Sta5}*{Theorem~3.11}. Part (iii) follows from the proof of \cite{Spi1}*{Proposition~2.10}, which computes the right LCM of $s$ and $t$ for $s\Cap t$.
\end{proof}

We recall that a core irreducible element in a right LCM monoid $S$ is an element $s$ such that any decomposition $s=ta$ with $t\in S$ and $a\in S_c$ is only possible if $a$ is an invertible element in $S$, see \cite{ABLS}. Further, $S$ is core factorable if every element $s$ admits an expression as the product of an element in $S_c$ and one that is core irreducible.

\begin{proposition}\label{prop:BS no core irreds}
For $c,d \in \IZ^\times$ with $|d|>1$, $BS(c,d)^+$ is core factorable if and only if $cd>0$. For $cd<0$, the identity is the only core irreducible element. The core irreducibles are always $\cap$-closed in $BS(c,d)^+$.
\end{proposition}
\begin{proof}
The forward direction of the first claim is established in \cite{ABLS}*{Proposition~5.10}, which ought to assume $c\geq 1, d>1$ in place of $c,d\geq 1, cd>1$. Now suppose $cd<0$. Take any $s\in BS(c,d)^+$ with  $\theta(s) >0$ and let $n \in \N$ be arbitrary, then the element $t=s\ib^{-n}$ is in $BS(c,d)^+$ and satisfies $t \neq s=t\ib^n$. Hence the identity is the only core irreducible element in $BS(c,d)^+$.  Note for completeness that when $cd>0$, the core irreducible elements are precisely the \emph{stems}, $t\ib^n$ with $n\in \N$, from \cite{CaHR}.  In particular, $BS(c,d)^+$ is not core factorable when $cd<0$ because the core is a proper submonoid. The last claim follows again from \cite{ABLS}*{Proposition~5.10} for $cd>0$, and is trivial for $cd<0$.
\end{proof}

\begin{remark}\label{rem:positivity breaking}
In view of Proposition~\ref{prop:BS no core irreds}, the Baumslag--Solitar monoids $BS(c,d)^+$ with $cd<0$ provide the first examples of right LCM monoids that are not core factorable. Hence the Zappa-Sz\'{e}p product $S_{ci} \bowtie S_c$ is a proper submonoid of $S$. The phenomenon behind this remarkable behaviour is a form of \emph{positivity breaking}: In the simplest case of $(1,-2)$, the monoid $BS(c,d)^+$ can be thought of as a variant of $BS(1,2) \cong \N \rtimes_2 \N$ which does not act by an endomorphism of the positive cone $\N\subset \Z$, but a composition of such an endomorphism with the flip on $\Z$ restricted to the domain $\N$. In abstract terms, we combine an injective endomorphism of a positive cone $P$ inside a group $G$ with an automorphism of $G$ that does not preserve $P$. This recipe opens the gates to a wide range of examples of (right LCM) semigroups with positivity breaking.
\end{remark}

\begin{remark}\label{rem:not Toeplitz}
During the workshop in Newcastle mentioned earlier, it was observed by Xin Li that $BS(c,d)^+$ with $cd<0$ are also the first examples of group embeddable monoids that do not admit a Toeplitz group embedding. The latter condition allows to regard semigroup crossed products by $S$ as full corners in group crossed products for the respective group, see \cite{CELY}*{Proposition~5.8.5}. The failure of the Toeplitz condition for the embedding of $BS(c,d)^+$  in the respective group is due to the positivity breaking for $cd<0$, which means that $(BS(c,d), BS(c,d)^+)$ is not quasi-lattice ordered.
\end{remark}

\begin{proposition}\label{prop:BS KMS}
For all nonzero integers $c$ and $d$ with $|d|>1$, $\beta_c$ equals $1$, and the equilibrium states on $(C^*(BS(c,d)^+),\sigma)$ are as follows:
\begin{enumerate}[(i)]
\item The ground states are parameterised by states on $C^*(\N)$.
\item The KMS$_\beta$-states for $\beta \in (1,\infty]$ are parameterised by normalised traces on $C^*(\Z)$.
\end{enumerate}
\end{proposition}
\begin{proof}
As $|d|>1$, $BS(c,d)^+$ admits a generalised scale given by $N_s = |d|^{\theta(s)}$ for $s\in BS(c,d)^+$. Note that in this case $\beta_c=1$. Parts (i) and (ii) are due to Proposition~\ref{prop:BS features}(i) combined with Proposition~\ref{prop:construction of ground states} and Proposition~\ref{prop:construction of KMS-states}, respectively.
\end{proof}

By \cite{CaHR}*{Proposition~7.1, Examples~7.3 and 7.4}, if $c\geq 1$ and $d \geq 2$, the KMS$_1$-state is unique precisely when $c\notin d\Z$. In the case $c\in d\Z$, the authors of \cite{CaHR} constructed two distinct KMS$_1$-states. Furthermore, they observed that when $c=d$ there is an isomorphism $\CQ(BS(c,d)^+) \cong \CO_d \otimes C^*(\Z)$ carrying the generating isometries $\pi(v_{\ib^{j}\ia})$ of the boundary quotient to $s_j\otimes 1$, where $s_1,\dots ,s_d$ are the universal generating isometries of $\CO_d$, and taking $\pi(v_{\ib^{d}})$ to the unitary $1\otimes u$, where $u$ is the generating unitary of $C^*(Z)$.

We will now  characterise uniqueness of the KMS$_1$-state in the case of arbitrary $c,d \in \Z^\times$ with $|d|>1$ and also determine the simplex of KMS$_1$-states for $c \in d\Z$. For this purpose, let us denote by $\tau_0$ the normalised trace on $C^*(\langle{\ib}\rangle^+) \cong C^*(\N)$ determined by $\tau_0(v_{\ib^k}) = \delta_{0,k}$ for $k \in \N$. By right cancellation in $BS(c,d)^+$, we have $\psi_{\beta,\tau_0}\circ \varphi = \tau_0$ for all $\beta>1$. Thus,  in the notation of Theorem~\ref{thm:restricting to tau}, we have $\tau_0 \in T_{>1}(C^*(\langle{\ib}\rangle^+))$.

\begin{proposition}\label{prop:BS KMS uniqueness}
Let $c,d \in \Z^\times, |d| >1$.
\begin{enumerate}[(i)]
\item There is a KMS$_1$-state $\psi_1$ determined by $\psi_1\circ \varphi = \tau_0$.
\item For $c=md$ with $m \in \Z^\times$, the simplex of KMS$_1$-states is parameterised by normalised traces $\tau$ on $C^*(\Z)$ such that $\tau(u^k) =\tau(u^{km^n})$ for all $k\in \N$ and $n\geq 1$.
\item The state $\psi_1$ is the unique KMS$_1$-state if and only if $c\notin d\Z$.
\end{enumerate}
\end{proposition}
\begin{proof}
Part (i) is an application of Theorem~\ref{thm:restricting to tau} as $\tau_0 \in T_{>1}(C^*(\langle{\ib}\rangle^+))$. For (ii) we will also apply Theorem~\ref{thm:restricting to tau}. We will therefore determine the constraints for normalised traces $\tau$ on $C^*(\langle{\ib}\rangle^+) \cong C^*(\N)$ arising from $\tau \in T_{>1}(C^*(\N))$. As traces on $C^*(\N)$ correspond to traces on $C^*(\Z)$, the result is phrased in terms of the latter. So let $c=md$ and $\beta >1$, and take $s \in BS(c,d)^+$ with $\theta(s) \geq 1$, that is, $s \notin S_c$. Given $k \in \N$, Proposition~\ref{prop:BS features}~(iii) implies that ${\ib}^ks \sim s$ if $k \in d\N$, and ${\ib}^ks \perp s$ otherwise. In the first case, the assumption $c=md$ implies that ${\ib}^ks= s{\ib}^{km^{\theta(s)}}$. Thus, if $\phi$ is a KMS$_1$-state, then by Remark~\ref{rem:reconstruction for beta=1} we obtain
\[\begin{array}{c}
\phi(v_{{\ib}^k}) = n^{-1} \sum\limits_{[s] \in N^{-1}(n)/_\sim} \phi(v_s^*v_{{\ib}^ks}^{\phantom{*}}) = \chi_{d\N}(k) \ \phi(v_{{\ib}^{km^\ell}})  \quad\text{ for all } n=|d|^\ell \in N(S),
\end{array}\]
where $l=\theta(s)$ varies in $\N$ as $s$ varies, and we have used that $v_s^*v_{{\ib}^ks}^{\phantom{*}}$ vanishes whenever ${\ib}^ks \perp s$.
Upon identifying the simplex of normalised traces of $C^*(d\N) \cong C^*(\N)$ with the one of $C^*(\Z)$, this is precisely the claimed condition in (ii). To prove that the condition parametrises KMS$_1$-states we need to show that every normalised trace $\tau$ on $C^*(v_{\ib^d})\subset C^*(BS(c,d)^+)$ with $\tau(v_{\ib^k}) =\tau(v_{\ib^{km^n}})$ for all $n\geq 1$ satisfies $\psi_{\beta,\tau}\circ \varphi = \tau$. To do so, we apply Lemma~\ref{lem:evaluation under state}~(i) to get
\[
\chi_{\tau,s}(v_{\ib^k}) = \tau\circ E(v_s^*v_{\ib^ks}^{\phantom{*}}) = \chi_{d\N}(k) \tau(v_{\ib^{km^{\theta(s)}}}).
\]
Hence $\chi_{\tau,s}(v_{\ib^k})=0$ unless $k \in d\N$, and its value depends not on $[s]$, but only on $\theta(s)$. Therefore, with the notation of Lemma~\ref{lem:construction of KMS-states}, we have  $\psi_{\tau,n} = \chi_{\tau,s}$ for every $s \in N^{-1}(n)$. With this observation we compute
\[\begin{array}{c}
 \psi_{\beta,\tau}(v_{\ib^k}) = \chi_{d\N}(k) \zeta(\beta)^{-1} \sum\limits_{\ell \geq 0} (|d|^\ell)^{1-\beta} \tau(v_{\ib^{km^\ell}}) = \tau(v_{\ib^k})
\end{array}\]
for all $k \in \N$. In view of Theorem~\ref{thm:restricting to tau}, this proves (ii).

For (iii), we note that $\alpha$ is almost free if $c\notin d\Z$, where the argument from  \cite{ABLS}*{Proposition~5.10(i)} applies verbatim. Thus $S$ is core regular by Proposition~\ref{prop:neg q-a vs. other conditions}, and therefore uniqueness of the KMS$_1$-state follows from Proposition~\ref{prop:uniqueness at 1 if neg q-a}. If $c=md$ for some $m \in \Z^\times$, then
\[
\tau_m (u^k) := |m|^{-1} \sum_{j \in \Z/m\Z} e^{2\pi i kj/m}
\] defines a normalised trace on $C^*(\Z)$ with $\tau_m(u^{\ell m}) = 1 \neq 0 = \tau_0(u^{\ell m})$ for all $\ell \in \Z^\times$. Upon identifying $\tau_m$ and $\tau_0$ with the corresponding normalised traces on $C^*(\langle\ib\rangle^+)$, Theorem~\ref{thm:restricting to tau} shows that $\tau_m$ and $\tau_0$ give rise to distinct KMS$_1$-states on $C^*(BS(c,d)^+)$.
\end{proof}

\begin{remark}\label{rem:BS KMS uniqueness}
As opposed to the case $c=d$, where all traces $\tau$ on $C^*(\Z)$ yield distinct KMS$_1$-states for $C^*(BS(c,d)^+)$, the case $c=-d$ has the constraint that $\tau$ needs to be real valued on $u^k, k \in \Z$ because Proposition~\ref{prop:BS KMS uniqueness}~(ii) states that $\tau(u^k) = \tau(u^{-k}) = \overline{\tau(u^k)}$. Therefore, the parity of $c$ and $d$ has an impact on the KMS-state structure.
\end{remark}

\section{Virtual group endomorphisms and self-similar group actions}\label{subsec:Heisenberg}

In this section we consider the problem of constructing and classifying KMS-states for natural dynamics on $C^*$-algebras associated to self-similar actions $(G,X)$. This has been done earlier, cf. \cite{LRRW} using a Toeplitz-Pimsner $C^*$-algebra model based on the Hilbert $C^*$-correspondence from \cite{Nek2}, and \cite{ABLS} using a monoidal Zappa--Sz\'{e}p product $X^*\bowtie G$ and its $C^*$-algebra as in \cite{BRRW}. Uniqueness of the $\kms_1$-state at critical value $\beta_c=1$ was linked in \cite{LRRW} to the property of $(G,X)$ called finite-state, and in \cite{ABLS}, it was phrased in terms of a condition called finite propagation.

Our approach here is to use the virtual endomorphism picture of self-similar actions as in \cite{Nek}.  Using this setup, we  illustrate with an example from \cite{BK} that the semigroup $C^*$-algebra of $(G, X)$ cannot witness the finite-state property itself, especially not through uniqueness of the $\kms_1$-state, see Corollary~\ref{cor:applying BK} and Example~\ref{ex:neg q-a for Heisenberg}.

In fact, noting that uniqueness of the $\kms_1$-state is actually a property of the virtual endomorphism rather than the particular associated self-similar action, we will show that the same is true for core regularity of the associated monoid. This is accomplished in Proposition~\ref{prop:almost surely regular} by providing a measure theoretic characterisation of core regularity.

Let $G$ be a discrete group. Recall from  \cite{Nek}*{Definition~2.7} that a \emph{self-similar action} $(G,X)$ is given by a finite alphabet $X$ (in at least two letters) together with an action  of $G$ on the free monoid $X^*$ in $X$ with the following property: For all $x \in X, g\in G$, there are $y\in X, h\in G$ such that
\[g(xw) = yh(w) \quad \text{for all } w \in X^*.\]
When the action of $G$ on $X^*$ is faithful, the elements $y\in X^*$ and $h\in G$ are uniquely determined. Then $h$ is referred to as the restriction of $g$ with respect to $x$, denoted by $g|_x$. We will assume henceforth that $(G,X)$ is faithful. Intuitively, we think of $g$ as a level-preserving automorphism of the rooted $\lvert X\rvert$-regular tree $X^*$, in which case $g|_x$ is the action of $g$ on the subtree attached at $x$. The recursive nature of self-similar actions of groups makes them extremely versatile, see \cite{Nek1} and the references therein.

There is a  close connection between self-similar actions and virtual endomorphisms of groups, see   \cite{Nek}. A \emph{virtual endomorphism} of a group $G$ is a homomorphism $\phi\colon H \to G$ where $H$ is a finite index subgroup of $G$. To every virtual endomorphism $(G,\phi, H)$ there is an  associated self-similar action $(G,X)$ with $\lvert X\rvert = [G:H]$ defined as follows: let $X$ be an alphabet with $[G:H]$ distinct letters. Choose a transversal $D$ for the left cosets in $G/H$ (called the \emph{digit set}), and label the elements as $D= \{q_x \mid x \in X\}$. The self-similar action associated to $(G,\phi,D)$ is then given by $(G,X)$ with relation $g(xw) = yg|_x(w)$ for $g \in G,x \in X$ and all $w \in X^*$, where $y \in X$ is the unique letter satisfying
\begin{equation}\label{eq:selfsimilar-from-virtual-definition}
  gq_x \in q_yH\text{ and }g|_x:=\phi(q_y^{-1}gq_x^{\phantom{1}}).
\end{equation}
Conversely, whenever $(G,X)$ is a faithful self-similar action, one can construct a virtual endomorphism $\phi$ of $G$ whose domain is the stabilizer group of a letter $x \in X$. If $(G,X)$ is also recurrent (self-replicating), then $(G,X)$ can be recovered as the self-similar action associated to $(G,\phi,D)$ for some digit set $D$, see \cite{Nek1}*{\S 2.8} for more details.

When associating a self-similar action to a virtual endomorphism $(G,\phi, H)$, the choice of digit set may affect key properties of the resulting self-similar action. In \cite{BK}*{4~Example} Bondarenko and Kravchenko give an example with $G$ equal  the discrete Heisenberg group where changing only one of the digits leads to self-similar actions with and without the finite-state property.  Nevertheless, all self-similar actions arising from a virtual endomorphism as in equation \eqref{eq:selfsimilar-from-virtual-definition} are topologically conjugate, that is, their actions on the respective spaces $X^\N$ of one-sided infinite words are, see \cite{Nek}*{Proposition~4.19}.

 Recall from \cite{BRRW} that to a faithful self-similar action $(G,X)$ there is an associated right LCM monoid $X^*\bowtie G$ called the Zappa-Szep product of $X^*$ and $G$. We show next that when $(G,X)$ arises from a virtual endomorphism $(G,\phi,H)$  the associated $C^*$-algebra $C^*(X^*\bowtie G)$ does not depend on the choice of digit set.

\begin{definition}\label{def:C*-alg of virt endom}
Given a virtual endomorphism $\phi$ of a discrete group $G$ with domain $H\subset G$, we let $C^*(G,\phi)$ be the universal unital $C^*$-algebra generated by a unitary representation $u$ of $G$ and an isometry $s_\phi$ subject to the relations
\begin{equation}\label{eq:relations for C*(G,phi)}
\begin{array}{c}
u_h s_\phi = s_\phi u_{\phi(h)} \text{ for all $h \in H$, and } \sum\limits_{[g] \in G/H} u_g^{\phantom{*}}s_\phi^{\phantom{*}}s_\phi^*u_g^* \leq 1.
\end{array}
\end{equation}
\end{definition}

\begin{thm}\label{thm:SSA a virt endom C*-algebras}
Let $\phi$ be a virtual endomorphism of $G$ with domain $H$, choose a digit set $D$ and let $(G,X)$ be the associated self-similar action. Then \[
u_g \mapsto v_{(\varnothing,g)},\, s_\phi \mapsto v_{(\varnothing,q_x^{-1})(x,1_G)},
\]
with $g\in G$ and $x \in X$ arbitrary, defines an isomorphism $C^*(G,\phi) \cong C^*(X^*\bowtie G)$.
\end{thm}
\begin{proof}
We use the identification $C^*(X^*\bowtie G) \cong \CT(G,X)$ from \cite{BRRW} and employ the notation for $\CT(G,X)$ from \cite{LRRW}, that is, we also write $u$ for the unitary representation $v_{(\varnothing,\cdot)}$ of $G$ and $s_x$ for the isometries $v_{(x,1_G)}, x \in X$. By \cite{LRRW}*{Proposition~3.2}, the relations for $\CT(G,X)$ are
\begin{equation}\label{eq:relations for T(G,X)}
\begin{array}{c}
u_g s_x = s_{g(x)} u_{g|_x} \text{ and } \sum\limits_{x \in X} s_x^{\phantom{*}}s_x^* \leq 1.
\end{array}
\end{equation}
First we check that $u$ and $t_{\phi} := u_{q_x^{-1}}s_x$ for $x \in X$ define a representation of $C^*(G,\phi)$. To begin with, we show that
$t_{\phi}$ does not depend on the choice of $x\in X$. Let $z\in X$. By \eqref{eq:relations for T(G,X)}, to have $u_{q_x^{-1}}s_x=u_{q_z^{-1}}s_z$ is equivalent to
\[
s_{(q_z^{\phantom{}}q_x^{-1})(x)}u_{(q_z^{\phantom{}}q_x^{-1})|_x}=s_z.
\]
This equality is satisfied if $(q_z^{\phantom{}}q_x^{-1})(x)=z$ and $(q_z^{\phantom{}}q_x^{-1})|_x=1_G$, and these identities hold because the recipe in \eqref{eq:selfsimilar-from-virtual-definition} implies that $z$ is the unique element in $X$ such that $(q_z^{\phantom{}}q_x^{-1})q_x\in q_zH$ and
\[
(q_z^{\phantom{}}q_x^{-1})|_x=\phi(q_z^{-1}(q_z^{\phantom{}}q_x^{-1})q_x)=\phi(1_G).
\]
Let $z$ be the unique element in $X$ such that $q_z\in H$. Note that we then have $q_z^{-1}(z)=z$ and $q_z^{-1}|_{z}=\phi(q_z^{-1})$. For every $h \in H$, we have $hq_x^{-1}q_x =h \in q_zH$, and hence
\begin{equation*}
\begin{array}{lcl}
u_ht_\phi &=& u_{hq_x^{-1}}s_x = s_{(hq_x^{-1})(x)}u_{(hq_x^{-1})|_x} \\
&=&s_zu_{\phi(q_z^{-1})}u_{\phi(h)} \\
&=& t_\phi u_{\phi(h)}.
\end{array}
\end{equation*}
This establishes the first relation from \eqref{eq:relations for C*(G,phi)}, and implies $u_h^{\phantom{*}}t_\phi^{\phantom{*}}t_\phi^*u_h^* = t_\phi^{\phantom{*}}t_\phi^*$ for all $h\in H$. We deduce that $u_g^{\phantom{*}}t_\phi^{\phantom{*}} t_\phi^*u_g^*$ does not depend on the choice of representative for $[g]\in G/H$, and since $D$ is a transversal for $G/H$, there is a bijection between $\{u_g^{\phantom{*}}t_\phi^{\phantom{*}} t_\phi^*u_g^*\mid [g]\in G/H\}$ and $\{s_x^{\phantom{*}}s_x^*\mid x\in X\}$. Therefore
\[\begin{array}{c}
\sum\limits_{[g] \in G/H} u_g^{\phantom{*}}t_\phi^{\phantom{*}}t_\phi^*u_g^*  =  \sum\limits_{x \in X} s_x^{\phantom{*}}s_x^* \leq 1.
\end{array}\]
This gives rise to a surjective $*$-homomorphism $C^*(G,\phi) \to \CT(G,X) \cong C^*(X^*\bowtie G)$. It is not hard to see that the assignment $u_g \mapsto u_g$ and $s_x \mapsto u_{q_x}s_\phi =: t_x$ defines a $*$-homomorphism on $\CT(G,X)$ that is inverse to the first map: the Toeplitz-Cuntz relation is clear. For $g \in G$ and  $x \in X$ there is a unique $y \in X$ such that $q_y^{-1}gq_x \in H$, hence by \eqref{eq:relations for C*(G,phi)} implies that $u_g t_x = u_{q_y} (u_{q_y^{-1}gq_x}s_\phi) = t_y u_{g|_x}$, as required in \eqref{eq:relations for T(G,X)}.
\end{proof}

After this prelude, let us proceed with the analysis of KMS-states. The time evolution on $C^*(X^*\bowtie G)$ is the one from \cite{ABLS}*{Proposition~5.8}, thus for  $t\in \R$, $g\in G$ and $x\in X$, so that in particular $\ell(x)=1$, we have
\[
\sigma_t(v_{(\varnothing,g)})=v_{(\varnothing,g)}\, \sigma_t(v_{(\varnothing,q_x^{-1})(x,1_G)})=|X|^{it}v_{(\varnothing,q_x^{-1})(x,1_G)}.
\]
The natural time evolution on $C^*(G,\phi)$ for a given virtual endomorphism $(G,\phi,H)$ is given by $\sigma_t(u_g)=u_g$ and $\sigma_t(s_\phi)=[G:H]^{it}s_\phi$ for all $t\in \R$ and $g\in G$.  The isomorphisms between $C^*(X^*\bowtie G)$, $\CT(G,X)$ and $C^*(G,\phi)$ from Theorem~\ref{thm:SSA a virt endom C*-algebras}  are equivariant with respect to the respective time evolutions.
Via the identification $C^*(X^*\bowtie G)\cong \CT(G,X)$, recall from
 \cite{LRRW}*{Theorem~7.3(3)} that for every finite-state self-similar group action $(G,X)$, the $C^*$-algebra $C^*(X^*\bowtie G)$ has a unique KMS-state at the critical inverse temperature, see also \cite{ABLS}*{Proposition~5.8(iii)}. As we shall now argue, core regularity improves the results in this direction. It not only covers new, non-finite state examples, but adds a new perspective in that it connects with so-called $G$-regular points of $(G,X)$.

For a self-similar group action $(G,X)$ with $G$ countable, consider the induced action of $G$ on the Cantor set $X^\N$ of right-infinite words in $X$ (equipped with the product topology). Let $g \in G$. Following the more recent terminology of \cite{Nek3}*{Definition~2.1} rather than \cite{Nek2}*{Definition~3.3}, a point $w \in X^\N$ is \emph{$g$-regular} if either $w$ is not fixed by $g$ or it lies in the interior of the set of fixed points for $g$. The second condition means that there is a finite word $v \in X^*$ with $w \in vX^\N$ such that $v \in A_{\lvert X\rvert^{\ell(v)}}^g$. A point that is not $g$-regular is \emph{$g$-singular}. Finally, a point is called \emph{$G$-regular} if it is $g$-regular for all $g\in G$.

Let $\mu$ denote the Borel probability measure on $X^\N$ given by the product measure of uniform distribution on $X$ at each stage, that is, $\mu[vX^\N] = \lvert X\rvert^{-\ell(v)}$ for every $v \in X^*$.

\begin{proposition}\label{prop:almost surely regular}
For a faithful self-similar action $(G,X)$ with countable group $G$, the following are equivalent:
\begin{enumerate}[(i)]
\item $X^*\bowtie G$ is core regular.
\item $\mu$-almost every point in $X^\N$ is $G$-regular.
\end{enumerate}
\end{proposition}
\begin{proof}
Let us fix $g\in G$ and note that $(S_n^g)_{n\geq 1}$ with $S_n^g:= \bigsqcup_{v \in F_{\lvert X\rvert^n}^g\setminus A_{\lvert X\rvert^n}^g} vX^\N$ forms a decreasing sequence of Borel subsets, see Lemma~\ref{lem:factorization formula for fixed non-absorbing elements}. We claim that the set of $g$-singular points coincides with $\bigcap_{n\geq 1}S_n^g = \lim_{n\to \infty} S_n^g$.

Suppose that $\omega$ is $g$-singular and $v \in X^n$ is the beginning of $\omega$ of length $n$. Then $v \in F_{\lvert X\rvert^n}^g$ as $g(\omega)=\omega$, but $v \notin A_{\lvert X\rvert^n}^g$ as there exists $\omega' \in vX^\N$ with $g(\omega') \neq \omega'$. Thus every $g$-singular point lies in $\lim_{n\to \infty} S_n^g$. Conversely, let $\omega \in \lim_{n\to \infty} S_n^g$. Then every beginning of $\omega$ is fixed, and thus $g(\omega) \in vX^\N$ for every beginning $v$ of $\omega$. As the cylinder sets in $vX^\N$ separate $\omega$ from every other point $\omega'$ ($X^\N$ is Hausdorff with the specified topology), we conclude that $g(\omega) = \omega$. Moreover, if $v$ is a beginning of $\omega$ of length $n$, then $v \in F_{\lvert X\rvert^n}^g\setminus A_{\lvert X\rvert^n}^g$, that is, $g(v)=v$ but $g_{|v}\neq 1_G$. By faithfulness of $(G,X)$, there is $v' \in X^*$ such that $g(vv') \neq vv'$. Thus every $\omega' \in vv'X^\N$ satisfies $g(\omega') \neq \omega'$, and hence $\omega$ is $g$-singular.

With the above claim, we easily calculate
\[\mu[\{g\text{-singular}\}] = \lim\limits_{n\to \infty} \mu[S_n^g] =\lim\limits_{n\to \infty} \frac{\lvert F_{\lvert X\rvert^n}^g\setminus A_{\lvert X\rvert^n}^g\rvert}{\lvert X\rvert^n}\]
for every $g \in G$. Taking into account that $\{G\text{-regular}\} = \bigcap_{g \in G}\{g\text{-regular}\}$ and that $G$ is countable, we therefore see that core regularity for $X^*\bowtie G$ is equivalent to $G$-regularity of $\mu$-almost all points in $X^\N$.
\end{proof}

\begin{remark}\label{rem:measure mu for right LCM with gs}
Every right LCM monoid $S$ with generalised scale $N$ admits an analogue of the measure $\mu$ on the spectrum $\widehat{\CD}$ of the diagonal subalgebra $\CD \subset C^*(S)$ generated by the commuting projections $e_{sS}, s\in S$. Its relevance in the context of KMS-states will be discussed in \cite{NS1}.
\end{remark}

\begin{remark}\label{rem:nqa indep of the digit set}
It follows from \cite{Nek1}*{Proposition~4.19} that for any two self-similar actions arising from the choice of two digit sets for a virtual endomorphism $(G,\phi)$, the respective actions $G\curvearrowright X^\N$ are topologically conjugate. In fact, the result includes an explicit description of the equivariant homeomorphism, say $h$, of $X^\N$, and we highlight that $h(vX^\N) =wX^\N$ with $\ell(w)=\ell(v)$ for every finite word $v$. This allows us to conclude that the probability measure $\mu$ appearing in Proposition~\ref{prop:almost surely regular} is $h$-invariant. Hence core regularity is a property of $(G,\phi)$, that is, it does not depend on the digit set.
\end{remark}

In view of Theorem~\ref{thm:SSA a virt endom C*-algebras} and Remark~\ref{rem:nqa indep of the digit set}, it seems that the natural approach to study the KMS-state structure of $(C^*(G\bowtie X^*),\sigma)$ is by transferring the question to  $(C^*(G,\phi),\sigma)$ for a virtual endomorphism $\phi$ that allows for a realization of $(G,X)$ via a suitably chosen digit set. Let us first note that injectivity of the virtual endomorphism simplifies the task:

\begin{lemma}\label{lem:independent of digit set}
Let $\phi$ be a virtual endomorphism of a group $G$ with domain $H$, $D$ a digit set for $(G,\phi, H)$ and $(G,X)$ the self-similar action associated to $(G,\phi,D)$. If $\phi$ is injective, then $X^*\bowtie G$ is right cancellative, and in particular there are no nontrivial absorbing words for all $g \neq 1_G$.
\end{lemma}
\begin{proof}
We have that  $X^*\bowtie G$ is right cancellative if and only if there do not exist $g'\in G\setminus\{1_G\}$ and $w \in X^*$ with $g'(w)=w, g'_{|w}=1_G$,  that is, $w \in A_{|X|^{\ell(w)}}^{g'}$. If such $g'$ and $w$ existed, then we may assume the same phenomenon occurs already for some $g \in G\setminus\{1_G\}$ and $x \in X$. But $1_G=g|_x = \phi(q_x^{-1}gq_x^{\phantom{1}})$ for $g\neq 1$ would contradict injectivity of $\phi$.%\footnote{Nadia: I changed this from $w \in A_{\ell(w)}^{g'}$. -Nico: Alright.}
\end{proof}

\begin{corollary}\label{cor:unique KMS1 for virt endom}
Suppose $\phi$ is a virtual endomorphism of a group $G$. If there exists a digit set $D$ for which the associated $X^*\bowtie G$ is core regular, for instance if $(G,X)$ is finite-state, then $(C^*(G,\phi),\sigma)$ has a unique KMS$_1$-state $\psi_1$. If $\phi$ is injective, then $\psi_1(u_g) = \delta_{g,1_G}$ for all $g\in G$.
\end{corollary}
\begin{proof}
Core regularity is implied by the finite-state property, see \cite{ABLS}*{Proposition~5.8(iii) and Lemma~9.5}, and Proposition~\ref{prop:uniqueness at 1 if neg q-a} implies the first claim.  If $\phi$ is injective, Lemma~\ref{lem:independent of digit set} shows that there are no non-trivial absorbing elements.
\end{proof}

This applies to a class exhibited in \cite{BK}*{Theorem~2}, to which we refer for a thorough explanation of the terminology.

\begin{corollary}\label{cor:applying BK}
Suppose $\phi$ is a surjective virtual endomorphism of a finitely generated torsion-free nilpotent group $G$ with trivial $\phi$-core. Let $\widehat{\phi}$ denote the differential of $\phi$ at the identity inside the Lie algebra of the Mal'cev completion of $G$. If the spectral radius of $\widehat{\phi}$ is at most $1$, and the cells for eigenvalues of modulus $1$ in the Jordan normal form of $\widehat{\phi}$ all have size $1$, then $(C^*(G,\phi),\sigma)$ has a unique KMS$_1$-state $\psi_1$.
\end{corollary}
\begin{proof}
By \cite{BK}*{Theorem~1}, there exists a digit set for $(G,\phi)$ whose associated self-similar action is finite-state, so that Corollary~\ref{cor:unique KMS1 for virt endom} applies.
\end{proof}

\begin{example}\label{ex:neg q-a for Heisenberg} Consider the \emph{discrete Heisenberg group}
\[G = \left\{ (x,y,z) := \scalebox{0.8}{$\begin{pmatrix} 1&x&z\\0&1&y\\0&0&1 \end{pmatrix}$} \mid x,y,z \in \Z\right\}.\]
For $m,n \in \Z^\times$, we consider the subgroup $H_{m,n} := \{ (x,y,z) \in G \mid x \in m\Z, y\in n\Z, z \in mn\Z\}$ of index $(mn)^2$ and the virtual endomorphism $\phi_{m,n}\colon H_{m,n} \to G$ given by $\text{diag}(1/m,1/n,1/mn)$, that is, $(mx,ny,mnz)\mapsto (x,y,z)$. In this case, $\widehat{\phi}_{m,n} = \text{diag}(1/m,1/n,1/mn)$. According to \cite{BK}*{Theorems~1 and 2}, the case $(m,n)=(1,2)$ allows for both finite-state and non-finite-state self-similar actions, which is demonstrated in \cite{BK}*{4~Example} for $X:= \{1,2,3,4\}$ by choosing $D_1 := \{0,e_2,e_3,e_2+e_3\}$ and $D_2 := \{e_1,e_2,e_3,e_2+e_3\}$, respectively.

 Corollary~\ref{cor:applying BK} applies, so that $(C^*(G,\phi_{m,n}),\sigma)$ has a unique KMS$_1$-state, whenever $|mn|>1$.

 It is known that the $G$-regular points for a faithful self-similar action of a countable group always form a dense $G_\delta$-set. It is natural to ask the following question: is there a countable discrete group $G$ that admits a faithful self-similar action $(G,X)$ with the property that the $G$-singular points have positive measure? Any such example is likely to be of great interest on its own as it shows quite distinct behaviour with respect to the two analogous statements from topology and ergodic theory.
\end{example}

\section{Shadowed natural numbers}\label{sec:unique despite unfaithful alpha}

We introduce now a class of right LCM monoids that are ultimatively a mock version of $\N\rtimes \N^\times$, but we restrict our scope to a single generator $\ic \in \N^\times$ to keep technicalities and notation to a minimum.  We replace $\N$ by the \emph{left-absorbing monoid} $L_1:= \langle \ia,\ib \mid \ia\ib=\ib^2\rangle^+$. If we think of $\ib$ as the generator of $\N$, then $\ia$ is merely a shadow that turns into $\ib$ as soon as it meets one on its right hand side.

For $n\geq 1$, let $L_n := \langle \ia,\ib \mid \ia\ib^n=\ib^{n+1}\rangle^+$ denote the \emph{left absorbing} monoid, see \cite{DDGKM}*{Reference Structure~8} for details. This is a right LCM monoid that is not right cancellative. In fact, every two elements in $L_n$ have a right common multiple, so $L_n$ is equal to its core, and hence does not admit a generalised scale. Nevertheless, we can use $L_1$, whose elements have a unique normal form $\ib^i\ia^j$ with $i,j \in \N$, to build an example with an interesting feature: For $m >1$, let
\[S_m := L_1 \rtimes_m \N = \langle \ia,\ib,\ic \mid \ia\ib=\ib^2, \ic\ia=\ia^m\ic, \ic\ib=\ib^m\ic\rangle^+.\]
Since the appearing endomorphism of $L_1$ is injective\footnote{Injectivity of such endomorphisms unfortunately only holds for $n=1$. For instance, $L_2$ has $\ia\ib\neq \ib^2$ but $\ic\ia\ib = \ib^{2m}\ic = \ic\ib^2$.}, $S_m$ is left cancellative, and we claim that it is right LCM. So let $s=\ib^{i_1}\ia^{j_1}\ic^{k_1}, t=\ib^{i_2}\ia^{j_2}\ic^{k_2} \in S_m$ with $k_1\leq k_2$. Then $s\Cap t$ if and only if $i_1+j_1 -(i_2+j_2) \in m^{k_1}\Z$. Suppose first that there are $r= \ib^{i_3}\ia^{j_3}\ic^{k_3}$ and $r'= \ib^{i_4}\ia^{j_4}\ic^{k_4}$ with $sr=tr'$. This yields
\begin{equation} \label{eq:left absorbing}
\ib^{i_1+j_1+m^{k_1}(i_3+j_3+m^{k_3})}\ic^{k_1+k_3} = sr\ib = tr'\ib = \ib^{i_2+j_2+m^{k_2}(i_4+j_4+m^{k_4})}\ic^{k_2+k_4},
\end{equation}
and as $k_2\geq k_1$, we deduce that $i_1+j_1 -(i_2+j_2) \in m^{k_1}\Z$. Conversely, if this condition holds, then there are  $j_3,j_4 \geq 0$ such that $i_1+j_1+m^{k_1}j_3 = i_2+j_2 +m^{k_2}j_4$, which yields $s\ib^{j_3}\ic^{k_2-k_1} = t\ib^{j_4}$.
In the case of $s\Cap t$, it is clear from \eqref{eq:left absorbing} that, given $r \in sS_m\cap tS_m$, we can always find $r' \in sS_m\cap tS_m$ with $r \in r'S_m$ such that the power of $\ic$ in $r'$ is $k_2$, i.e.~the maximum of $k_1$ and $k_2$. To conclude that $S_m$ is right LCM, we now only need to figure out that this case allows for exactly one minimal choice of a word in $\ia$ and $\ib$, which we leave to the reader as an exercise.

Borrowing terminology from \cite{Sta5}*{Definitions~3.2,3.4}, it follows that
\begin{enumerate}[(i)]
\item $(S_m)_c$ coincides with $L_1$;
\item for $s=\ib^{i_1}\ia^{j_1}\ic^{k_1}$ and $t=\ib^{i_2}\ia^{j_2}\ic^{k_2} \in S_m$, $s \sim t$ holds if and only if $k_1 = k_2$ and $i_1+j_1 = i_2+j_2 (\text{mod } m^{k_1})$;
\item an element $s=\ib^i\ia^j\ic^k$ is noncore irreducible if and only if $k=1$;
\item the core graph $\Gamma(S_m)$ is the empty graph on $m$ vertices.
\end{enumerate}
As $S_m$ is noncore factorable (and has balanced factorisation for trivial reasons), \cite{Sta5}*{Theorem~3.11} implies that $S_m$ has a generalised scale $N$ determined by $\ia,\ib \mapsto 1, \ic\mapsto m$.
In particular, $\beta_c=1$ holds for all $m>1$.

\begin{proposition}\label{prop:unfaithful, but neg q-a}
For every $m>1$, the action $\alpha\colon L_1 \curvearrowright S_m$ is not faithful, but $S_m$ is core regular, and therefore $(C^*(S_m),\sigma)$ has a unique KMS$_1$-state.
\end{proposition}
\begin{proof}
Due to $\ia\ib=\ib^2$, every $[s] \in S/_\sim$ is represented by an element of the form $\ib^j\ic^k$ with $j>0$. Thus $\alpha$ is not faithful at $\ia \neq \ib$. For $\ib^{i_1}\ia^{j_1}, \ib^{i_2}\ia^{j_2} \in L_1$, observation (ii) from the list above entails that, for $n=m^k$ large enough so that $i_1+j_1,i_2+j_2 < m^k$, the set $F_n^{\ib^{i_1}\ia^{j_1},\ib^{i_2}\ia^{j_2}}$ can only be nonempty if $i_1+j_1=i_2+j_2$. But in this case we get
\[\ib^{i_1}\ia^{j_1}s \sim \ib^{i_1}\ia^{j_1}s\ib = \ib^{i_1+j_1+m^k}s = \ib^{i_2}\ia^{j_2}s\ib \sim \ib^{i_2}\ia^{j_2}s\]
for every $s \in N^{-1}(n)$, so that $A_n^{\ib^{i_1}\ia^{j_1},\ib^{i_2}\ia^{j_2}} = F_n^{\ib^{i_1}\ia^{j_1},\ib^{i_2}\ia^{j_2}} = N^{-1}(n)/_\sim$ holds. Thus $S_m$ is core regular and the KMS$_1$-state is unique by Proposition~\ref{prop:uniqueness at 1 if neg q-a}.
\end{proof}

This example shows the strength of the approach taken in Proposition~\ref{prop:uniqueness at 1 if neg q-a}, as both results for uniqueness of the KMS$_1$-state in \cite{ABLS} require $\alpha$ to be faithful.

\section{Summable quasi-absorption in the absence of almost freeness}\label{subsec:q-a but not almost free}

 There are only few explicit examples of monoids $S$ for which the critical inverse temperature $\beta_c$ is strictly larger than the minimal value $1$ and uniqueness of the $\kms_\beta$-state is known to hold for all $\beta$ in the critical interval $[1,\beta_c]$, namely $\N\rtimes \N^\times$ through \cite{LR2}, $R\rtimes R^\times$ for rings of integers $R$ in a number field, see \cite{CDL} or \cite{Nes}. In the case of $\N\rtimes \N^\times$, it was demonstrated in \cite{ABLS} that uniqueness is due to almost freeness of the core action $\alpha$. This is no longer true even for $\Z\rtimes \Z^\times$, but this example is covered by the involved arguments of \cite{CDL}. Here, we show that one can simply apply our conditions of core regularity to get uniqueness for $\Z\rtimes\Z^\times$, see Proposition~\ref{prop:ZxZtimes}.

 In a similar way, we treat an example built on a shift space over the finite cyclic groups involving a Frobenius automorphism in one direction, see Proposition~\ref{prop:shift space with permutations}. Thereby, we provide two examples of right LCM monoids that are core regular and summably core regular, while $\alpha$ fails to be almost free. In view of Proposition~\ref{prop:r canc: neg q-a forces alpha faithful}, uniqueness of the KMS$_1$-state is only possible due to failure of right cancellation.  The reason why the monoids in these examples are core regular is that whenever $\alpha_a=\alpha_b$ for $a,b \in S_c, a \neq b$, then this is witnessed already by absorption in $S$, that is, for every $s \in S$ there are $c,d \in S_c$ with $asc=bsd$.

The first example relates to $\N\rtimes \N^\times$ of \cite{LR2}: $S:=\Z\rtimes\Z^\times$ with $S_c=S^*=\Z\rtimes \Z^*$.

\begin{proposition}\label{prop:ZxZtimes}
The right LCM monoid $S$ has a generalised scale with $\beta_c=2$. The action $\alpha\colon S_c \curvearrowright S/_\sim$ is faithful, but not almost free. The monoid $S$ is core regular and summably core regular. In particular, $(C^*(S),\sigma)$ has a unique KMS$_\beta$-state $\psi_\beta$ for each $\beta \in [1,2]$ determined by $\psi_\beta(v_g) = \delta_{g,1}$ for $g \in S^*$.
\end{proposition}
\begin{proof}
Noting that $S$ arises as the semidirect product of the algebraic dynamical system $(\Z,\theta,\Z^\times)$ (with $\theta$ being multiplication), we deduce from \cite{ABLS}*{Proposition~5.11} that $S$ has a generalised scale, but $\alpha\colon \Z\rtimes \Z^* \curvearrowright S/_\sim$ is not almost free. Similar to the case of the monoid $\N\rtimes \N^\times$, we have $\beta_c=2$ for $S$. Since $S$ is cancellative, there are no nontrivial absorbing elements, see Proposition~\ref{prop:r canc: neg q-a forces alpha faithful}. Moreover, $S_c=S^*$ is a group, so Remark~\ref{rem:neg q-a for core being a group} reduces the question of core regularity to studying the fixedpoint sets $F_p^g = \{[(h,p)] \mid h+p\Z \in \Z/p\Z, \alpha_g([(h,p)]) = [(h,p)]\}$ for $g \in \Z\rtimes\Z^*\setminus\{0\}$ and $p \in N(S)=\N^\times$. We have that $\alpha_g([(h,p)])=[(h,p)]$ for $g=(g',\pm1)$ if and only if $g'\pm h-h \in p\Z$. The case of $g \in \Z\times\{1\}, g\neq (0,1)=1_S$ resembles the case for $\N\rtimes\N^\times$: The map $\alpha_g$ only has fixedpoints of the form $[(h,p)]$ with $p\leq |g'|$. So the action $\alpha$ restricted to $\Z\times\{1\}$ is almost free. Hence  both types of core regularity hold by Proposition~\ref{prop:neg q-a vs. other conditions}. Thus let $g=(g',-1)$. We need to determine the number of cosets $h+pZ \in \Z/p\Z$ with $g'-2h \in p\Z$.
\begin{description}
\item[Case $p \notin 2\Z$] then $2$ defines a permutation of $\Z/p\Z$, so $h+p\Z$ with $g'-2h \in p\Z$ is uniquely determined and $\lvert F_p^g\rvert = 1$.
\item[Case $p \in 2\Z, g' \notin 2\Z$] there is no such $h$, so $F_p^g=\emptyset$.
\item[Case $p \in 2\Z, g' \in 2\Z$] then $g'-2h \in p\Z$ is equivalent to $h \in g'/2 +p/2\Z\setminus p\Z \sqcup g'/2 +p\Z$, so that $\lvert F_p^g\rvert = 2$.
\end{description}
Thus we see that $\lvert F_p^g\rvert \leq 2$ for all $g\in \Z\rtimes\{-1\}, p \in \N^\times$, and conclude that $S$ is core regular. In addition, Corollary~\ref{cor:suff for q-a} with $E=\emptyset$ yields summable core regularity for these remaining cases.  The claimed uniqueness of the KMS$_\beta$-state and its form for $\beta$ inside the critical interval now follow from Proposition~\ref{prop:uniqueness at 1 if neg q-a} and Proposition~\ref{prop:uniqueness in the crit interval by s q-a}.
\end{proof}

The second example is inspired by \cite{ABLS}*{Example~5.13}. We decided to keep it explicit instead of aiming for the greatest generality, with the hope that the interested reader will be able to extract the essential ingredients for the recipe. Let us denote the set of all primes in $\N$ by $\FP$, and fix $q\in \FP$. For $p \in \FP$, we let $G_p := \bigoplus_\N \Z/p\Z$ as a group. For each $p \in \FP\setminus \{q\}$, the Frobenius map $g \mapsto g^q$ defines an automorphism of $\Z/p\Z$, and hence an automomorphism of the group $G_p$ in the natural way. Acting as the identity on $G_q$, we obtain an automorphism $f_q$ of the group
\[G:= \bigoplus_{p \in \FP} G_p = \{ (g_{p,n})_{p \in \FP,n \in \N} \mid g_{p,n} \in \Z/p\Z\}.\]
In addition to $f_q$, we will let $\N^\times$ act on $G$  via the shift maps $\Sigma$ with $\Sigma_p|_{G\setminus G_p} = \id$ and $(g_1,g_2,g_3,\ldots) \mapsto (0,g_1,g_2,g_3,\ldots)$ on $G_p$. Each $\Sigma_p$ is an injective endomorphism of $G$ whose image has index $p$. As $\Sigma_p$ commutes with $f_q$, it is easy to check that $(G,\Sigma\times f_q, \N^\times\times \Z)$ forms an algebraic dynamical system in the sense of \cite{BLS2}. Hence the induced semidirect product $S:= G \rtimes_\theta (\N^\times \times \Z)$ is a cancellative right LCM monoid with $S_c=S^*=G\rtimes_{f_q} \Z$.

\begin{proposition}\label{prop:shift space with permutations}
For every $q \in \FP$, the right LCM semigroup $S$ has a generalised scale with $\beta_c = 2$. The action $\alpha\colon S^* \curvearrowright S/_\sim$ is faithful, but not almost free. The monoid $S$ is core regular and summably core regular. In particular, $(C^*(S),\sigma)$ has a unique KMS$_\beta$-state $\psi_\beta$ for each $\beta \in [1,2]$ determined by $\psi_\beta(v_g) = \delta_{g,1}$ for $g \in S^*$.
\end{proposition}
\begin{proof}
The proof of Proposition~\ref{prop:ZxZtimes} can be transferred to this setting. Therefore, we restrict ourselves to the analysis of $F_p^g$ for $g \in G\rtimes_{f_q} \Z$ and $p \in \N^\times$. For $p\in \FP$ and $g=(g',1)$, we find that $\alpha_g [(h,p)] = [(h,p)]$ if and only if $g'_{p,1}h_{p,1}^q = h_{p,1}$. As $\Z/p\Z$ is abelian, this is equivalent to $h_{p,1}^{q-1} =g'_{p,1}$. If $q-1$ and $p$ are relatively prime, the map $\tilde{h} \mapsto \tilde{h}^{q-1}$ defines an automorphism of $\Z/p\Z$, so that $h_{p,1}$ and hence $[(h,p)]$ is uniquely determined in this case. As $p \in \FP$, this will be true for almost all $p \in \FP$. More precisely, it fails for the finite set $E_q$ given by $q$ and the primes dividing $q-1$. Then we conclude that $\lvert F_p^g\rvert = 1$ for all $p \in \FP\setminus E_q$ and all $g \in S^*$. Since $\FP\setminus E_q \neq \emptyset$, this shows core regularity. Faithfulness of $\alpha$ now follows from right cancellation and Proposition~\ref{prop:r canc: neg q-a forces alpha faithful}. Summable core regularity follows from Corollary~\ref{cor:suff for q-a}. The claimed uniqueness of the KMS$_\beta$-state and its form for $\beta$ inside the critical interval now follow from Proposition~\ref{prop:uniqueness at 1 if neg q-a} and Proposition~\ref{prop:uniqueness in the crit interval by s q-a}.
\end{proof}

\end{document}